\documentclass[11pt]{article}
\usepackage{amssymb,amsmath,amsfonts,amsthm,mathtools,enumerate,color}
\usepackage[toc,page,title,titletoc,header]{appendix}
\usepackage{graphicx,subfig}

\usepackage{paralist}

\graphicspath{ {./figures/} }

\usepackage{indentfirst}
\usepackage{multicol}
\usepackage{booktabs}
\usepackage{url}
\usepackage{epstopdf} 

\usepackage{hyperref}
\hypersetup{
    colorlinks=true, 
    linktoc=all,     
    linkcolor=blue,  
}
\numberwithin{equation}{section}

\newtheorem{Theorem}{Theorem}[section]
\newtheorem{Lemma}[Theorem]{Lemma}
\newtheorem{Proposition}[Theorem]{Proposition}
\newtheorem{Definition}[Theorem]{Definition}
\newtheorem{Corollary}[Theorem]{Corollary}
\newtheorem{Algorithm}{Algorithm}

\newtheorem{Remark}{Remark}
\newtheorem{Assumption}{Assumption}

 \def\p{\partial} \def\nb{\nonumber}
\def \Vh0{\stackrel{\circ}{V}_h} \def\to{\rightarrow}
   
\def\Om{\Omega}  \def\om{\omega} 
 
\newcommand{\q}{\quad}

\def\l{\label}  \def\f{\frac}  \def\fa{\forall}
\def\b{\beta}  \def\a{\alpha} 

\def\eps{\varepsilon}

 \def\t{\times}  
\def\ms{\medskip}  

\def\p{\partial}

  \def\x{{\bf x}} 
\def\y{{\bf y}}

\def\cA{\mathcal{A}}
\def\cB{\mathcal{B}}

\def\cE{\mathcal{E}}
\def\cF{\mathcal{F}}
\def\cG{\mathcal{G}}
\def\cH{\mathcal{H}}
\def\cI{\mathcal{I}}

\def\cK{\mathcal{K}}
\def\cL{\mathcal{L}}
\def\cM{\mathcal{M}}

\def\cQ{\mathcal{Q}}

\def\cT{\mathcal{T}}
\def\cU{\mathcal{U}}
\def\cX{\mathcal{X}}
\def\cY{\mathcal{Y}}

\def\bA{{\textbf{A}}}

\def\N{{\mathbb{N}}}
\def\bP{\mathbb{P}}

\def\R{{\mathbb R}}
\def\Z{{\mathbb{Z}}}

\newcommand{\ex}{\mathbb{E}}

\newcommand{\tr}{\textnormal{tr}}
\DeclareMathOperator*{\argmin}{arg\,min}

\newcommand{\lc}
{\mathrel{\raise2pt\hbox{${\mathop<\limits_{\raise1pt\hbox
{\mbox{$\sim$}}}}$}}}

\newcommand{\gc}
{\mathrel{\raise2pt\hbox{${\mathop>\limits_{\raise1pt\hbox{\mbox{$\sim$}}}}$}}}

\newcommand{\ec}
{\mathrel{\raise2pt\hbox{${\mathop=\limits_{\raise1pt\hbox{\mbox{$\sim$}}}}$}}}

\def\bb{\begin{equation}} \def\ee{\end{equation}}
\def\bbn{\begin{equation*}} \def\een{\end{equation*}}

\def\beqn{\begin{eqnarray}}  \def\eqn{\end{eqnarray}}

\def\beqnx{\begin{eqnarray*}} \def\eqnx{\end{eqnarray*}}

\def\bn{\begin{enumerate}} \def\en{\end{enumerate}}

\def\bd{\begin{description}} \def\ed{\end{description}}

\makeatletter
\newenvironment{tablehere}
  {\def\@captype{table}}
  {}
\newenvironment{figurehere}
  {\def\@captype{figure}}
  {}
\makeatother

\begin{document}

\title{A penalty scheme and policy iteration for nonlocal HJB variational inequalities with monotone drivers}
\author{
Christoph Reisinger\thanks{Mathematical Institute, University of Oxford, United Kingdom ({\tt christoph.reisinger@maths.ox.ac.uk, yufei.zhang@maths.ox.ac.uk})}
\and
Yufei Zhang\footnotemark[2]
}
\date{}

\maketitle


\noindent\textbf{Abstract.} 
We propose a class of numerical schemes for 
nonlocal HJB variational inequalities (HJBVIs) with monotone drivers.
The solution and free boundary of the HJBVI are constructed from a sequence of penalized equations, for which a continuous dependence result is derived and  the penalization error is estimated. 
The penalized equation is then discretized by a class of semi-implicit monotone approximations.
We  present a novel analysis technique for the well-posedness of the discrete equation, and demonstrate the convergence of the scheme, which subsequently gives a constructive proof for the existence of a solution to the penalized equation and variational inequality. We further propose an efficient iterative algorithm with local superlinear convergence for solving the discrete equation. 
 Numerical experiments are presented for an optimal investment problem under ambiguity and a recursive consumption-portfolio allocation problem.


\ms

\medskip

\section{Introduction}
In this article, we consider a nonlocal Hamilton-Jacobi-Bellman variational inequality (HJBVI) of the following form: 
\begin{align}\l{eq:hjbvi}
0&=F(t,x,u,Du,D^2u,\{K^{\a} u\}_{\a\in\bA},\{B^{\a} u\}_{\a\in\bA})\\
&=\begin{cases}
\min\big\{u-\zeta, u_t+\inf_{\a\in \bA}\big(-L^{\a} u-f^{\a}(t,x,u,(\sigma^{\a})^T Du,B^{\a} u)\big)\big\},  & (t,x)\in \cQ_T ,\nb\\
u(0,x)-g(x), & x\in\R^d,\nb
\end{cases}
\end{align}
with the operators $L^{\a} \coloneqq A^{\a} +K^{\a} $ and $B^{\a}$ satisfying for $\phi\in C^{1,2}(\bar{\cQ}_T)$ that
\begin{align}
A^{\a} \phi(t,x)&=\f{1}{2} \tr(\sigma^{\a} (t,x)(\sigma^{\a} (t,x))^TD^2\phi(t,x))+b^{\a} (t,x) \cdot D\phi,\l{eq:a}\\
K^{\a} \phi(t,x)&=\int_{E}\big(\phi(t,x+\eta^{\a} (t,x,e))-\phi(t,x)-\eta^{\a} (t,x,e)\cdot D\phi(t,x)\big)\,\nu(de),\l{eq:k}\\
B^{\a} \phi(t,x)&=\int_{E}m\big(\phi(t,x+\eta^{\a} (t,x,e))-\phi(t,x)\big)\gamma(t,x,e)\,\nu(de),\l{eq:b}
\end{align}
where we denote $\cQ_T= (0,T]\t \R^d$ and $E=\R^n\setminus\{0\}$.
The  function $f$, called the driver of \eqref{eq:hjbvi}, is monotone, possibly non-Fr\'{e}chet-differentiable, and of arbitrary growth in its third component (precise conditions will be specified later).

Such equations  extend the classical HJBVIs with linear drivers, i.e., $f(\a,t,x,y,z,k)\equiv \ell(\a,t,x)-r(t,x)y$, and play an important role in modern  finance, including the following: 
models for American options in a market with constrained portfolios \cite{karoui1994,karoui1997},  recursive utility optimization problems \cite{epstein1992}, and robust pricing and risk measures under probability model uncertainty \cite{roger2006,quenez2013}. We remark that imposing merely monotonicity assumptions on the drivers allows us to consider several important non-smooth drivers stemming from robust pricing \cite{roger2006, karoui2009, quenez2013} and non-Lipschitz drivers arising in stochastic recursive control \cite{kraft2013,pu2017}, while including an extra nonlinearity in the operator $B^\a$ enables us to incorporate ambiguity in the jump processes \cite{roger2006}. 
As  the solution to \eqref{eq:hjbvi} is in general not known analytically, it is important to construct effective and robust numerical schemes for solving these fully nonlinear equations.

To the best of our knowledge,  there is no published numerical scheme covering the generality of \eqref{eq:hjbvi}. However, there is a vast literature on monotone approximations for local HJB equations (e.g., \cite{crandall1984,barles1991,kushner1992,debrabant2012}) and on monotone finite-difference quadrature schemes for nonlocal HJB equations (e.g., \cite{biswas2010diff,biswas2017}). 
For works covering specific extensions, we refer the reader to \cite{hintermuller2002,
jakobsen2003,ito2006,huang2012,howison2013} for penalty approximations to  variational inequalities, to \cite{bokanowski2009, witte2011} for an application of policy iteration together with penalization to solve HJB obstacle problems with linear drivers, to \cite{dumitrescu2018} for schemes to HJB obstacle problems with Lipschitz drivers based on piecewise constant  policy time stepping, and to \cite{xu2017} for applying policy iteration to solve (finite-dimensional) static HJB equations with Fr\'{e}chet-differentiable concave drivers and finite control sets. 

In this paper, we shall construct a class of monotone schemes for solving \eqref{eq:hjbvi} with a monotone (possibly non-Fr\'{e}chet-differentiable) driver and a compact  set of controls. 
Note that monotonicity of the scheme is crucial, since it is well-known that non-monotone schemes may fail to converge or even converge to false ``solutions'' \cite{debrabant2012}. By Godunov's Theorem \cite{godunov1959}, in general, one can expect a monotone scheme to be at most first-order accurate.

We emphasize that 
the non-Lipschitz setting of the drivers prevents us from 
adopting the standard Banach fixed-point arguments (see e.g.~\cite{dumitrescu2018}) to establish  the well-posedness and stability of the discrete approximations of \eqref{eq:hjbvi}, and hence new analysis techniques are required.
Moreover, although in practice one can obtain a finite-dimensional equation by localizing the scheme on a bounded  domain, it is  important to analyze the discrete equation in an infinite-dimensional setting, since numerical solutions will behave  asymptotically similar to that from the infinite-dimensional equation as one refines the mesh or enlarges the  domain.
The  non-differentiablility of the driver in $y$ and its nonlinear dependence on $z$ and $k$ also introduce substantial difficulties in designing efficient iterative algorithms for solving the discrete equations.

The main contributions of this work are as follows:
\begin{itemize}
\item We formulate a penalty approximation to \eqref{eq:hjbvi} with a  monotone driver, and establish a continuous dependence estimate for the penalized equation, independent of the penalty parameter. We shall demonstrate that as the penalty parameter tends to infinity, the solution of the penalized equation converges to the solution of \eqref{eq:hjbvi} monotonically from below, at a  rate depending explicitly on the regularity of the obstacle, 
which extends the results in \cite{jakobsen2003} to nonlocal equations with monotone drivers and time-dependent obstacles. 
These convergence results further lead us to a convergent approximation of the free boundary of \eqref{eq:hjbvi}, which to our best knowledge is new, even in the classical cases with linear drivers.

\item We propose a class of semi-implicit monotone approximations to the penalized equations, 
which enjoy a stablity condition independent of the penalty parameter. 
We further present a novel analysis technique for the well-posedness of  the resulting (infinite-dimensional) discrete equation by constructing Lipschitz approximations of the monotone driver via smoothing and truncation.
The convergence of the scheme is demonstrated, which subsequently gives a constructive proof for the existence of a  bounded viscosity solution to the penalized equations and the HJBVI \eqref{eq:hjbvi}.

\item 
For practical implementations, we propose an efficient iterative algorithm for a localized 
discrete equation with a slantly differentiable driver, 
and demonstrate the local superlinear convergence for the value functions, which extends the results obtained in the cases with linear drivers (see \cite{bokanowski2009, witte2011}) or with Fr\'{e}chet-differentiable concave drivers and finite control sets \cite{xu2017}.
A novel convergence result of control strategies in the Hausdorff metric is established.
We further estimate the control discretization error caused by numerical approximations of the continuous controls, e.g.~by piecewise linearization.

\item  Numerical examples for an optimal investment  problem for jump-diffusion models under ambiguity and a recursive consumption-portfolio allocation problem
with stochastic volatility models
are included to investigate the convergence order of the scheme with respect to different discretization parameters.
\end{itemize}

We organize this paper as follows.  Section \ref{sec:assumption}
gives  standard definitions and assumptions on the HJBVI \eqref{eq:hjbvi}. We shall propose a penalty approximation to the HJBVI in Section \ref{sec:penalty} and study its  convergence properties.
Then we derive a class of fully discrete monotone schemes for the penalized equations in Section  \ref{sec:scheme_construct}, and establish their convergence  in Section  \ref{sec:scheme_analysis}. A Newton-type iterative method with local superlinear convergence is constructed in Section \ref{sec:policy} to solve the resulting discrete equations. 
Numerical examples for an optimal investment   problem  under ambiguity and a recursive consumption-portfolio allocation problem are presented in Section \ref{sec:numerical} to illustrate the effectiveness of our algorithms. 

\section{Main assumptions and preliminaries}\l{sec:assumption} 
In this section, we state our main assumptions on the coefficients of \eqref{eq:hjbvi} and introduce related concepts of solutions.
We start by collecting some useful notation which is needed frequently throughout this work. 

For any given function  $g:\R^{d_1}\mapsto \R^{d_2}$, we define by $g^+\coloneqq \max(g,0)$ and  $g^-\coloneqq \max(-g,0)$ the (component-wise) positive  and  negative part of $g$, respectively. Also 
for any given signed measure $\nu$, we denote by $\nu^+$ and $\nu^-$, respectively,
the positive part and the negative part in the Jordan decomposition of $\nu$, and by $|\nu|=\nu^++\nu^-$ the total variation of $\nu$. Moreover, for any given positive measures $\nu_1$ and $\nu_2$, we define their maximum by 
$$
\nu_1 \vee \nu_2\coloneqq \bigg(\f{d\nu_1}{d(\nu_1+\nu_2)}\vee \f{d\nu_2}{d(\nu_1+\nu_2)}\bigg) (\nu_1+\nu_2),
$$
where the derivatives denote the corresponding Radon-Nikodym derivatives.
Finally, for a function $f:\bar{\cQ}_T\to \R^{d_1\t d_2}$ we define the following (semi-)norms:
$$
|f|_0=\sup_{(t,x)\in \bar{\cQ}_T} |f(t,x)|,\q |f|_{1}=\sup_{t\in [0,T], x,x'\in \R^d,x\not=x'} \f{|f(t,x)-f(t,x')|}{|x-x'|}, \q \|f\|_1=|f|_0+|f|_{1},
$$
which extend naturally to time-independent functions or vectors.

Now we  turn to the  standing assumptions on the coefficients of the HJBVI \eqref{eq:hjbvi}:

\begin{Assumption}\l{assum:mono}
Let $\bA$ be a  compact subset of a metric space. Moreover, there exists  $C>0$ and  $\mu\in \R$ such that  it holds for any $(\a, e,t, x,z,k)\in \bA\t E\t \bar{\cQ}_T\t \R^d\t \R$ and $y,y'\in \R$ that:
\begin{enumerate}[(1)]
\item $b^{\a}, \sigma^{\a},\eta^{\a}$ are  continuous in $\a,t$, and $\gamma,\zeta$ are continuous in $t$, which satisfy the following estimates: $g(x)\ge \zeta(0,x)$, $\gamma(t,x,e)\ge 0$, and 
$$
\|b^{\a}\|_1+\|\sigma^{\a}\|_1+\|\zeta\|_1+\|g\|_1\le C, \q\!\! \|\eta^{\a}(\cdot,\cdot,e)\|_1+|\gamma(\cdot,\cdot,e)|_0\le C(1 \wedge |e|),\q\!\! |\gamma(\cdot,\cdot,e)|_1\le C(1 \wedge |e|^2).
$$
\item   $f:\bA\t \bar{\cQ}_T\t\R\t \R^{d}\t\R\to \R$ is a continuous function satisfying the  properties:
\begin{enumerate}
\item (Boundedness.) $|f^{\a}(t,x,0,0,0)|\le C$.
\item (Monotonicity.) The mapping $y\mapsto f^{\a}(t,x,y,z,k)$ is monotone  in the sense that there exists a continuous increasing function $\varphi:[0,\infty)\to [0,\infty)$ such that
\begin{align}
&(y-y')(f^\a(t,x,y,z,k)-f^\a(t,x,y',z,k))\le \mu |y-y'|^2, \l{eq:y_mono}\\
 &|f^\a(t,x,y,z,k)|\le |f^\a(t,x,0,0,0)|+\varphi(|y|)+C(|z|+|k|) , \l{eq:f_growth}
\end{align}
and $k\mapsto f^{\a}(t,x,y,z,k)$ is non-decreasing.
\item (Lipschitz continuity.) 
$f$ is  Lipschitz continuous in $x, z$ and $k$ with the constant $C$, uniformly in $\a, t$ and $y$.
\end{enumerate}
\item $m:\R\to \R$ is Lipschitz continuous and non-decreasing with $m(0)=0$.
\end{enumerate}
\end{Assumption}

\begin{Remark}\l{remark:strict_monotone}
Although our discussions focus on the cases with coefficients $b^{\a},\sigma^{\a},\zeta,g$ bounded in $x$, similar results and analysis are valid for coefficients with polynomial growth as well. 
Moreover, as pointed out in \cite{jakobsen2005}, there is no loss of generality by assuming  $f$ is strictly monotone in $y$ with $\mu<0$ in \eqref{eq:y_mono}
(this can be seen by carrying out an exponential time scaling of the solution).
%

Finally, we remark that for any  Lipschitz continuous and non-decreasing function $\psi:\R\to \R$, in particular the functions $m$ and $k\mapsto f^{\a}(t,x,u,p,k)$, one has
$\psi(x)-\psi(y)\le C(x-y)^+$ for any $x,y\in \R$, which will be used frequently in our subsequent analysis.
\end{Remark}

We emphasize that Assumption \ref{assum:mono} only requires $f$ to be monotone in $y$ (up to an additive linear function) and allows the coefficients $\sigma^\a$ and $\eta^\a$ to vanish at certain points. Therefore, due to lack of regularization from a Laplacian or fractional Laplacian operator,
the  solutions of \eqref{eq:hjbvi} are typically not smooth, and we shall understand the equation in the viscosity sense \cite{jakobsen2005}:

\begin{Definition}[Viscosity solution]\l{definition:viscosity}
An upper (resp.~lower) semicontinuous function $u$ is said to be a viscosity subsolution (resp.~supersolution) of \eqref{eq:hjbvi} if and only if 
for any point $\x_0=(t_0,x_0)$ and for any $\phi\in C^{1,2}(\bar{\cQ}_T)$ such that 
$\phi(\x_0)=u(\x_0)$ and $u-\phi$ attains its global  maximum (resp.~minimum) at $\x_0$, one has
\begin{align*}
&
F_*(\x_0,u(\x_0),D\phi(\x_0),D^2\phi(\x_0),\{K^{\a} \phi(\x_0)\}_{\a\in\bA},\{B^{\a} \phi(\x_0)\}_{\a\in\bA})\le 0\\
\big(resp. \q &
F^*(\x_0,u(\x_0),D\phi(\x_0),D^2\phi(\x_0),\{K^{\a} \phi(\x_0)\}_{\a\in\bA},\{B^{\a} \phi(\x_0)\}_{\a\in\bA})
\ge 0\big).
\end{align*}
%
%
A continuous function is a viscosity solution of the  HJBVI \eqref{eq:hjbvi} if it is both a  viscosity sub- and supersolution.
\end{Definition}

We will demonstrate that \eqref{eq:hjbvi} admits a unique bounded solution under Assumption \ref{assum:mono}. The uniqueness follows directly from the comparison principle which we establish below (see Remark \ref{remark:comparison}), while a continuous bounded solution can be explicitly constructed through discrete approximations (see Remark \ref{remark:soln_penal}).

\section{Penalty approximations for the HJBVI}\l{sec:penalty}

In this section, we shall propose a penalty approximation for the HJBVI \eqref{eq:hjbvi}, which is an extension of the ideas used for local HJB obstacle problems (with linear drivers) in \cite{jakobsen2003, witte2011} and for American options in \cite{howison2013}. 

For any given parameter $\rho\ge 0$, we shall consider the following penalized problem:
\begin{align}\l{eq:hjb_penal}
0&=F^\rho(\x,u,Du,D^2u,\{K^{\a} u\}_{\a\in\bA},\{B^{\a} u\}_{\a\in\bA})\\
&=\begin{cases}
u^\rho_t+\inf_{\a\in \bA}\big(-L^{\a} u^\rho-f^{\a}(\x,u^\rho,(\sigma^{\a})^T Du^\rho,B^{\a} u^\rho)\big)-\rho(\zeta-u^\rho)^+,  & \x\in \cQ_T ,\nb\\
u^\rho(0,x)-g(x), & x\in\R^d,\nb
\end{cases}
\end{align}
which will be interpreted in the viscosity sense similar to Definition \ref{definition:viscosity} by virtue of the possible  degeneracy of the equation. 

In the following, we shall focus on the uniqueness of viscosity solutions and investigate their dependence on the coefficients. The proof for the existence of  solutions will be deferred to Section \ref{sec:scheme}, where we will construct continuous solutions of \eqref{eq:hjb_penal} through numerical schemes and demonstrate they are bounded independent of the penalty parameter $\rho$ (see Remark \ref{remark:soln_penal}).


The next theorem presents a continuous dependence estimate for the solutions of the penalized equation, which quantifies the stability properties of solutions with respect to the coefficients. As the reader will  see immediately, this estimate not only implies the uniqueness of viscosity solutions, but also enables us to derive the 
convergence rate of the penalty approximation to the  HJBVI \eqref{eq:hjbvi} and construct convergent approximations for the free boundary. 
Moreover, it will be used in Sections \ref{sec:scheme} and \ref{sec:policy} to estimate the discretization errors. 

The proof of this estimate  follows essentially along the lines of the proof of \cite[Theorem~4.1]{jakobsen2005}, with  extra technicalities arising from 
the nonlinearities of $f^\a$ and  $B^\a$.
We  include a detailed proof in Appendix \ref{sec:appendix} for the convenience of the reader.

\begin{Theorem}\l{thm:conts}
Consider \eqref{eq:hjb_penal}  with two sets of coefficients  $\{b^\a_i,\sigma^\a_i, \eta^\a_i,\nu_i, \zeta_i,f_i\}_{i=1,2}$, which satisfy Assumption \ref{assum:mono} with the same  $C$, $\mu$ and $\varphi$. Let  
$u_1$ (resp.~$u_2$) be a bounded subsolution (resp.~supersolution) to  \eqref{eq:hjb_penal} with $i=1$ (resp.~$i=2$), then it holds for any $(t,x)\in \bar{\cQ}_T$ that
\begin{align}
 u_1(t,x)-u_2(t,x)\le &
| ( u_1(0,\cdot)-u_2(0,\cdot))^+|_0+\sup_{\a\in\bA}|f_1^\a(\cdot,\cdot,\cdot,0,0)-f_2^\a(\cdot,\cdot,\cdot,0,0)|_{\bar{\cQ}_T\t 
 [-\varphi(|u_2|_0),\varphi(|u_2|_0)]}\nb\\
&+|\zeta_1-\zeta_2|_0+ C\sup_{\a\in\bA}\bigg(|\sigma_1^\a-\sigma_2^\a|_0^{\f{1}{2}}+|b^\a_1-b^\a_2|_0^\f{1}{2} \l{eq:conts}\\
&+\bigg|\int_{E} |\eta^\a_1-\eta^\a_2|^2\,(\nu_1\vee\nu_2) (de)\bigg|^\f{1}{4}_0+\bigg|\int_{E}\max(|\eta^\a_1|^2, |\eta^\a_2|^2) \,|\nu_1-\nu_2|(de)\bigg|^\f{1}{4}_0\bigg).\nb
\end{align}

\end{Theorem}

An immediate consequence of the above continuous dependence estimate is the strong comparison principle for the penalized equation,
which implies  the uniqueness of viscosity solution to \eqref{eq:hjb_penal} in the class of bounded continuous functions.

\begin{Corollary}
Let $u$ and $v$ be a bounded  subsolution and  supersolution to \eqref{eq:hjb_penal}, respectively, with $u(0,x)\le v(0,x)$, then it holds under Assumption \ref{assum:mono} that $u(t,x)\le v(t,x)$ for all $(t,x)\in \bar{\cQ}_T$.
\end{Corollary}

\begin{Remark}\l{remark:comparison}
One can establish  the same estimate \eqref{eq:conts} for the  HJBVI \eqref{eq:hjbvi} by adapting the arguments for Theorem \ref{thm:conts} (see  \cite{jakobsen2003} for a discussion on the classical local HJB obstacle problems), and hence deduce similar comparison principle and uniqueness result  for the obstacle problem \eqref{eq:hjbvi}.

\end{Remark}

\begin{Remark}[Modulus of continuity]\l{rmk:modulus}
Compared to the results in \cite{jakobsen2005} that the solutions to classical HJB/Isaacs equations depend Lipschitz continuously on the local terms (in sup-norm) and  nonlocal terms (in $L^2$ norm), the solutions to \eqref{eq:hjbvi} and \eqref{eq:hjb_penal} are only H\"{o}lder continuous with respect to the coefficients with exponent $1/2$, mainly due to the additional Lipschitz nonlinearity of $f$ on $Du$ and $B^\a u$  (not the nonlinearity on $u$). 
The same  modulus of continuity has been demonstrated for a simpler case in \cite{dumitrescu2018} using probabilistic arguments. In the case where 
the nonlinearity of $f$ on $Du$ and $B^\a u$ admits particular structures (e.g.~the nonlinearity  can be expressed in a Hamiltonian form), 
one can possibly recover the standard Lipschitz dependence.
\end{Remark}

We now proceed to study the convergence of the penalized equation \eqref{eq:hjb_penal} to the HJBVI \eqref{eq:hjbvi}. The following theorem illustrates the monotonicity of $u^\rho$ in the penalty parameter $\rho$.

\begin{Theorem}\l{thm:mono_rho}
Suppose Assumption \ref{assum:mono} holds.   
Let  $u$ and $u^\rho$ be the viscosity solutions to, respectively, \eqref{eq:hjbvi} and \eqref{eq:hjb_penal} with parameter $\rho\ge 0$. Then it holds  for any $\rho_1\le \rho_2$ that $ u^{\rho_1}\le u^{\rho_2}\le u$.
\end{Theorem}
\begin{proof}
We start with the following two important observations, which can be established directly from the definitions: 
\begin{inparaenum}[(1)]
\item If $u^\rho$ is a subsolution to \eqref{eq:hjb_penal} with any $\rho\ge 0$, then $u^\rho$ is a subsolution to \eqref{eq:hjbvi};
\item If $\rho_1\le \rho_2$ and $u^{\rho_2}$ is a supersolution  to \eqref{eq:hjb_penal} with the parameter $\rho_2$, then 
$u^{\rho_2}$ is a supersolution to \eqref{eq:hjb_penal} with the parameter $\rho_1$.
\end{inparaenum}
Then the comparison principles for \eqref{eq:hjbvi} and \eqref{eq:hjb_penal}
enable us to conclude the desired results.
%
%
%
%
%
%
\end{proof}

The next result asserts the convergence rate of the penalized equation  to the HJBVI, which depends on the regularity  of the obstacle as observed in \cite{jakobsen2003,witte2011}.
\begin{Theorem}\l{thm:conv_rho}

Let  $u$ and $u^\rho$ be the viscosity solution to \eqref{eq:hjbvi} and \eqref{eq:hjb_penal}.
Suppose  Assumption \ref{assum:mono} holds and the obstacle $\zeta$ is H\"{o}lder continuous in $t$ with exponent $\mu\in (0,1]$, then there exists a constant $C_0>0$, independent of the penalty parameter $\rho$, such that
\bb\l{eq:penalty_conv}
0\le u(\x)-u^\rho(\x)\le C_0\rho^{-\min(\mu,\f{1}{2})}, \q \x\in \bar{\cQ}_T.
\ee
If we further assume  $\zeta\in C^{1,2}_b(\bar{Q}_T)$, then we have
\bb\l{eq:penalty_conv_smooth}
0\le u(\x)-u^\rho(\x)\le C_0/\rho, \q \x\in \bar{\cQ}_T.
\ee
\end{Theorem}
\begin{proof}
The proof  follows precisely the steps in the arguments for \cite[Theorem 2.1]{jakobsen2003}, together with the continuous dependence estimate
as ascertained by Theorem \ref{thm:conts}. The main steps are establishing \eqref{eq:penalty_conv_smooth} for regular obstacles, smoothing a general obstacle with the following mollifiers:
$$\varrho_\eps(t,x)=\eps^{-(d+\f{1}{1-\mu})} \varrho\big(\eps^{-\f{1}{1-\mu}}t,\eps^{-1}x\big),$$
where $\varrho$ is a positive smooth function supported in $\{0<t<1\}\t\{|x|<1\}$ with mass one, 
 and balancing the approximation error in the two
cases.
\end{proof}

We end this section with a convergent approximation of the  free boundary, $\Gamma=\{\x \in \bar{\cQ}_T\mid u(\x)=\zeta(\x)\}$,
 of the HJBVI \eqref{eq:hjbvi}  using the solution of  penalized equations. 
 Suppose the estimate $0\le u(\x)-u^\rho(\x)\le C_0\rho^{-\mu}$ holds for some constants $C_0>0$ and $\mu\in (0,1]$,
 we then define for each $\rho>0$ the set 
 \bb\l{eq:free_approx}
 \Gamma_\rho=\{\x\in \bar{\cQ}_T\mid \zeta(\x)-C_0\rho^{-\mu}\le u^\rho(\x)\le \zeta(\x)\}.
 \ee
It follows directly from the  estimates for $ u^\rho$ and $u$ that $\Gamma\subset \Gamma_\rho$ for all $\rho>0$.
The next result demonstrates that
$\Gamma_\rho$ in fact converges to $\Gamma$ in terms of the Hausdorff metric.

\begin{Theorem}
It holds for any given compact set $K\subset \bar{\cQ}_T$ that
$$
\lim_{\rho\to \infty}d_\cH(\Gamma_\rho \cap K,\Gamma\cap K)=\lim_{\rho\to \infty} \sup_{\y\in \Gamma_\rho \cap K}\inf_{\x\in \Gamma\cap K}|\x-\y|= 0.
$$
\end{Theorem}
\begin{proof}
Suppose the statement does not hold for a given compact set $K$, then there exist a constant $\eps>0$ and sequences $\{\rho_n\}$ and $\{\y_n\}=\{t_n,y_n\}$ such that $\rho_n\to \infty$, $\y_n\in \Gamma_{\rho_n} \cap K$ and $\inf_{\x\in \Gamma\cap K}|\x-\y_n|\ge\eps$. By passing to a subsequence, we can assume that $\y_n\to \y^*\in K$ and
$\inf_{\x\in \Gamma\cap K}|\x-\y^*|\ge\eps$, which implies that $\y^*\not\in \Gamma$.
However, the definition of $\Gamma_{\rho_n}$ gives
\begin{align*}
u(\y^*)-\zeta(\y^*)&=u(\y^*)-u(\y_n)+u(\y_n)-u^{\rho_n}(\y_n)+u^{\rho_n}(\y_n)-\zeta(\y_n)+\zeta(\y_n)-\zeta(\y^*)\\
&\le u(\y^*)-u(\y_n)+C_0\rho_n^{-\mu}+0+\zeta(\y_n)-\zeta(\y^*)\to 0,
\end{align*}
as $n\to \infty$, which together with the fact that $u\ge \zeta$ implies $\y^*\in \Gamma$, and hence a contradiction.
\end{proof}

\section{Discrete approximations for penalized equations}\l{sec:scheme}
In this section, we  propose a class of semi-implicit monotone approximations for solving the penalized equation \eqref{eq:hjb_penal} with a fixed penalty parameter $\rho\ge 0$.~We shall construct  the schemes in Section \ref{sec:scheme_construct} and perform their analysis in Section \ref{sec:scheme_analysis}. In order to  derive more accurate estimates for the truncation error and the stability condition of the schemes, 
throughout this section we shall impose the following condition on the L\'{e}vy measure:
\begin{Assumption}\l{assum:density}
The L\'{e}vy measure $\nu$ admits a density $k(e)$ with the following estimate: it holds for some constants $C>0$ and $\kappa\in [0,2)$  that
\bb\l{eq:density}
0\le k(e)\le C|e|^{-n-\kappa}, \q |e|<1, \; e\in E=\R^n\setminus\{0\}.
\ee
\end{Assumption}
We emphasize that Assumption \ref{assum:density} is  imposed for the sake of preciseness, and will only be used  in Section \ref{sec:scheme}.
Without the above estimate, one can still establish the consistency and stability of the schemes, but with more  pessimistic results (see for example \cite{dumitrescu2018}).

\subsection{Semi-implicit numerical methods}\l{sec:scheme_construct}
In this section, we shall derive semi-implicit monotone approximations for  \eqref{eq:hjb_penal}. We recall that it is crucial to construct a monotone discretization, since in general non-monotone schemes may fail to converge or even converge to false ``solutions" (see \cite{debrabant2012}).
For simplicity, we focus on the  uniform spatial grid   $\{x_i\}_i=h\Z^d$ on $\R^d$ and a time partition $\{t_n\}_{n=0}^N$ with 
$ \max_{n}|t_{n+1}-t_n|=\Delta t$, 
but similar results  are valid for unstructured nondegenerate  grids as well.

We start with the approximation of the nonlocal operators by truncating the small jumps of the L\'{e}vy measure and compensating it with an additional diffusion term as suggested in \cite{dumitrescu2018}. More precisely, for any $r\in (0,1)$, we introduce the truncated L\'{e}vy measure $\nu_r(de)=1_{|e|>r}\nu(de)$ and the modified diffusion coefficient $\sigma_r^\a$ with $\sigma_{r,ij}^\a=\sigma_{ij}^\a$ for $i\not =j$ and 
\bb\l{eq:small_diff}
\sigma^{\a}_{r,ii}(t,x)=\bigg((\sigma^{\a}_{ii}(t,x))^2+\int_{|e|<\eps}|\eta^{\a}_i(t,x,e)|^2\, \nu(de)\bigg)^{1/2}, \q i=1,\ldots, d, \; (t,x)\in \bar{\cQ}_T.
\ee
Then we shall consider a modified version of \eqref{eq:hjb_penal}
with the following  modified operators:
\begin{align}
A^{\a}_r \phi(t,x)= &\f{1}{2} \tr(\sigma_r^{\a} (t,x)(\sigma_r^{\a} (t,x))^TD^2\phi)+\big(b^{\a} (t,x)-\int_{|e|>r} \eta^{\a} (t,x,e)\,\nu(de)\big)\cdot D\phi,\l{eq:ar}\\
K^{\a}_r \phi(t,x)=&\int_{|e|>r}\big(\phi(t,x+\eta^{\a} (t,x,e))-\phi(t,x)\big)\,\nu(de),\l{eq:kr}\\
B^{\a}_r \phi(t,x)=&\int_{|e|>r}m\big(\phi(t,x+\eta^{\a} (t,x,e))-\phi(t,x)\big)\gamma(t,x,e)\,\nu(de),\l{eq:br}
\end{align}
for all $(t,x)\in \bar{\cQ}_T$ and test functions $\phi\in C^{1,2}(\bar{\cQ}_T)$. It is clear that these approximations 
 are consistent with \eqref{eq:hjb_penal} in the sense that for any  $\x=(t,x)\in \bar{\cQ}_T$,
\bb\l{eq:r_consistent}
|A^{\a}_r \phi(\x)+K^{\a}_r \phi(\x)-A^{\a} \phi(\x)-K^{\a}\phi(\x)|+|B^{\a}_r \phi(\x)-B^{\a}\phi(\x)|\le C\int_{|e|<r}|e|^2\, \nu(de)\le C r^{2-\kappa}.
\ee
In fact, suppose $u^\rho$ and $u^\rho_r$ solve the original and modified penalized equation respectively, one can  deduce from Theorem \ref{thm:conts} and 
Assumption \ref{assum:density}  that for any $(t,x)\in \bar{\cQ}_T$ we have:
\begin{align*}
| u_r(t,x)-u_r^\rho(t,x)|\le & C\sup_{\a\in\bA}\bigg(\big|\int_{|e|<r}(|\eta^\a(t,x,e)|^2 \nu(de)\big|^\f{1}{4}_0\bigg)\le O(r^{(2-\kappa)/4}).
\end{align*}
Since we consider a penalized equation \eqref{eq:hjb_penal} with a general nonlinear $f$, the above estimate may not be optimal for  equations with convex structures (see Remark \ref{rmk:modulus} for details).

The nonlocal operators $K^{\a}_r$ and $B^{\a}_r$ are then approximated by a combination of interpolation and 
quadrature rules as in \cite{debrabant2012,dumitrescu2018}.
Let $\cI_h$ be  a second-order positive interpolation operator on the spatial grid $\{x_i\}$, for instance the  linear or multilinear interpolations, such that for all $x\in \R^d$, 
\begin{align}\l{eq:error_intp}
\cI_h[\phi](x)=\sum_{m\in \Z^d}\phi(x_m)\omega_m(x;h),\q 
|\phi(x)-\cI_h[\phi](x)|\le Ch^2|D^2\phi|_0, 
\end{align}
where $\{\omega_m(x;h)\}_m$ are some basis functions satisfying $0\le \omega_m(x;h)\le 1$, $\sum_m\om_m=1$, $\omega_m(x_i;h)=\delta_{mi}$ and $\textrm{supp}\, \om_m\subset B(x_m,2h)$, we shall approximate the nonlocal terms \eqref{eq:kr} and \eqref{eq:br} by
\begin{align}
K^{\a}_{r,h} \phi(t_n,x_i)=&\int_{|e|>r}\cI_h[\phi(t_n,x_i+\cdot)-\phi(t_n,x_i)](\eta^{\a} (t_n,x_i,e))\,\nu(de),
\l{eq:Kh_approx}\\
B^{\a}_{r,h} \phi(t_n,x_i)=&\int_{|e|>r}m\big(\cI_h[\phi(t_n,x_i+\cdot)-\phi(t_n,x_i)](\eta^{\a} (t_n,x_i,e))\big)\gamma(t_n,x_i,e)\,\nu(de)\nb\\
=&\int_{|e|>r}m\big(\sum_{j\in \Z^d}\om_j(\eta^{\a} (t_n,x_i,e);h)[\phi(t_n,x_i+x_j)-\phi(t_n,x_i)]\big)\gamma(t_n,x_i,e)\,\nu(de), \l{eq:Bh}
\end{align}
which in practice can be further evaluated by using consistent quadrature rules with positive weights, such as Gauss methods of appropriate order.

We remark that using \eqref{eq:error_intp}, one can express \eqref{eq:Kh_approx} in  the following monotone form: 
\bb\l{eq:Kh}
K^{\a}_{r,h} \phi(t_n,x_i)=\sum_{j\in \Z^d}k^{\a, n}_{r,h,j,i}[\phi(t_n,x_i+x_j)-\phi(t_n,x_i)],\q
k^{\a, n}_{r,h,j,i}=\int_{|e|>r}\om_j(\eta^{\a} (t_n,x_i,e);h)\,\nu(de), 
\ee
while, as we will see in the subsequent analysis, the approximation \eqref{eq:Bh} is closely related to the following coefficients 
\bb\l{eq:brh}
b^{\a, n}_{r,h,j,i}=\int_{|e|>r}\om_j(\eta^{\a} (t_n,x_i,e);h)\gamma(t_n,x_i,e)\,\nu(de),\q j\in \Z^d,
\ee
where we assume without loss of generality that $k^{\a, n}_{r,h,0,i}=b^{\a, n}_{r,h,0,i}=0$ for all $i\in \Z^d$.

The boundedness of $\eta$ implies that the sums in \eqref{eq:Kh} and \eqref{eq:Bh} are finite. Moreover, one can deduce from \eqref{eq:error_intp} that for any fixed $r\in( 0,1)$, these approximations are consistent with the truncation error
\begin{align}\l{eq:KB_consistent}
\begin{split}
&|K^{\a}_{r,h} \phi(t_n,x_i)-K^{\a}_r \phi(t_n,x_i)|\le C|D\phi^2|_0h^2\Gamma_1(r,\kappa),\\
&|B^{\a}_{r,h} \phi(t_n,x_i)-B^{\a}_r \phi(t_n,x_i)|\le C|D\phi^2|_0h^2\Gamma_2(r,\kappa),
\end{split}
\end{align}
with some constant $C$ independent of $r$ and $h$, and 
\begin{align}
\Gamma_1(r,\kappa)=\int_{|e|>r} \nu(de)&\le \begin{cases} -\log r & \textnormal{if $\kappa=0$},\\ r^{-\kappa} &\textnormal{if $\kappa>0$}, \end{cases} \l{eq:gamma1}\\
\Gamma_2(r,\kappa)=\int_{|e|>r}(1\wedge |e|) \nu(de)&\le\begin{cases}  1 & \textnormal{if $\kappa\in [0,1)$},\\ -\log r &\textnormal{if $\kappa=1$}, \\ r^{1-\kappa} &\textnormal{if $\kappa\in (1,2)$},\end{cases} \l{eq:gamma2}
\end{align}
where we have used the density estimate \eqref{eq:density}.
Since Godunov's Theorem in \cite{godunov1959} asserts that 
one  in general can expect a monotone scheme to be  at most first order accurate,  in the following, we shall choose $r=\max(h^{1/\kappa},h)$ to ensure the truncation error \eqref{eq:KB_consistent} to be of the magnitude $O(h)$.

We now estimate the summations  of coefficients $k^{\a, n}_{r,h,j,i}$ and $b^{\a, n}_{r,h,j,i}$, which will be essential for the stability of the scheme. The property  $\sum_j \om_j=1$ leads immediately to the estimate
\bb\l{eq:sumkb}
\sum_j k^{\a, n}_{r,h,j,i}\le \Gamma_1(r,\kappa),\q \sum_j b^{\a, n}_{r,h,j,i}\le C\Gamma_2(r,\kappa), \q \fa i\in \Z^d, \; n=0,\ldots, N,
\ee
where $\Gamma_1$ and $\Gamma_2$ are defined as in \eqref{eq:gamma1} and  \eqref{eq:gamma2}, respectively.
It is worth pointing out that other upper bounds of these summations can been derived  using the approach in \cite{biswas2017}. In fact, suppose the basis functions have the property that $|D\om_j|_0\le C/h$, then we can deduce that
\begin{align*}
\sum_{j\not=0}k^{\a, n}_{r,h,j,i}=\sum_{j\not=0}\int_{|e|>r}\om_j(\eta^{\a} (t_n,x_i,e);h)-\om_j(0;h)\,\nu(de)\le \f{C}{h}\int_{|e|>r} |\eta^{\a} (t_n,x_i,e)|\nu(de)\le \f{C}{h}\Gamma_2(r,\kappa).
\end{align*}
However, the relation $r\ge h$ clearly implies that  \eqref{eq:sumkb} always gives a sharper upper bound.

We now proceed to consider  the modified local operator $A^\a_r$, which will be approximated by  a consistent and monotone scheme $A_{r,h}^\a$, such that for any test function $\phi$ we have
\begin{align}
|A^\a_r \phi-A^\a_{r,h}\phi| &\le C|D^2\phi|_0 h\Gamma_2(r,\kappa),\l{eq:A_consistent}\\
A^\a_{r,h}\phi(t_n,x_i)&=\sum_{j\in \Z^d} l^{\a, n}_{r,h,j,i}[\phi(t_n,x_j)-\phi(t_n,x_i)], \l{eq:Ah}
\end{align}
with some constant $C$ independent of $r$ and $h$, $\Gamma_2(r,\kappa)$ defined as in \eqref{eq:gamma2}, and coefficients $l^{\a, n}_{r,h,j,i}\ge 0$  for all $i,j\in \Z^d$ and $n$. 
The construction of numerical approximations with the above properties has been discussed  thoroughly in \cite{biswas2010diff}. In particular, one can adopt the  standard schemes of Kushner in \cite{kushner1992} if the diffusion coefficient is diagonally dominant, and use the semi-Lagrangian scheme in \cite{debrabant2012} if the coefficient $\tilde{\sigma}^a(\tilde{\sigma}^a)^T$ is not diagonally dominant.

Finally, we construct numerical approximations for the Lipschitz nonlinearity of $f$ on $Du$. For simplicity, we shall focus on the Lax-Friedrichs  numerical flux, but it is straightforward to extend our schemes and analysis to other  Lipschitz  numerical fluxes, for instance the Godunov flux, which are monotone and  consistent with $f$ (see \cite{crandall1984,osher1991}). 

Let $U_i^n$ be the discrete approximation of the solution to \eqref{eq:hjb_penal} at the node $(t_n,x_i)$, we  denote by 
 $\Delta^{(l)}_+U^n_i$ (resp.~$\Delta^{(l)}_-U^n_i$) the
 one-step forward (resp.~backward) difference of $U$ along the $l$-th coordinate for each $l = 1,\ldots,d$, and by $\Delta U^n_i =
(\Delta^{(1)}_+U^n_i +\Delta^{(1)}_-U^n_i ,\ldots, \Delta^{(d)}_+U^n_i +\Delta^{(d)}_-U^n_i )^T$ the central difference of $U$  at the node $(t_n,x_i)$.
Then for any given $(y,k)\in \R\t \R$,  the Lax-Friedrichs numerical flux is given by:
\bb\l{eq:lax}
\bar{f}^\a(t_n,x_i,y,\Delta U^n_i,k)\coloneqq f^\a(t_n,x_i,y,\sigma_r^{\a}(t_n,x_i)^T\f{\Delta U^n_i}{2h},k)+\sum_{l=1}^d \f{\theta}{\lambda}\bigg( \f{\Delta^{(l)}_+U^n_i -\Delta^{(l)}_-U^n_i}{h}\bigg),
\ee
where  $\lambda=\Delta t/h$ and $\theta>0$ is a prescribed parameter.

With all these spatial approximations in hand, we are ready to write the fully-discrete scheme for \eqref{eq:hjb_penal}. We shall adopt an implicit timestepping for the local term $A^\a_{r,h}$ and an explicit timestepping for the nonlocal term $K^\a_{r,h}$. This enables us to enjoy a less restrictive stability condition than that for  fully explicit schemes and avoid solving the dense system resulting from the integral operator. For the nonlinear terms,  we shall perform implicit timestepping for the $u$ term and explicit timestepping for $Du$ and $B^au$. As we will see later, by taking advantage of the monotonicity of the driver and the penalty term on $u$, our scheme can ensure stability with a less restrictive time stepsize, especially for a large penalty parameter $\rho$. 
Therefore, our semi-implicit scheme shall read as: $U^0_i=g(x_i)$ for all $i\in \Z^d$ and for any given $n=0,\ldots,N-1$:
\begin{align}\l{eq:hjbi_d}
0&=G_h(t_{n+1},x_i,U_i^{n+1},\{U^{b+1}_a\}_{(a,b)\not=(i,n)})\\
&=
\inf_{\a\in \bA}\bigg(\f{U^{n+1}_i-U^n_i}{\Delta t}-A^{\a}_{r,h} U^{n+1}_i-K^\a_{r,h}U^n_i-\tilde{f}^{\a}(t_n,x_i,U^{n+1}_i,\Delta U_i^n,B^{\a}_{r,h} U^n_i)\bigg), \q i\in\Z^d, \nb
\end{align}
where we wrote $\tilde{f}^\a(t,x,y,z,k)= \bar{f}^\a(t,x,y,z,k)+\rho(\zeta(t,x)-y)^+$ with $\bar{f}$ defined as in \eqref{eq:lax}.

\subsection{Well-posedness and convergence analysis}\l{sec:scheme_analysis}

In this section, we shall establish the well-posedness of  the discrete equation \eqref{eq:hjbi_d} and perform its convergence analysis, 
which subsequently leads us to a constructive proof for the existence of bounded solutions to the penalized equation \eqref{eq:hjb_penal} and the HJBVI \eqref{eq:hjbvi}. 
We emphasize that the non-Lipschitz dependence of $f$ on $y$ requires novel analysis techniques for the well-posedness and stability of schemes, which are essentially different from the fixed-point arguments in most existing works (see e.g. \cite{biswas2010diff,debrabant2012}). We remark that throughout this section we shall assume without loss of generality that $f$ is strictly monotone in $y$ with $\mu<0$ (see Remark \ref{remark:strict_monotone}).

We start by recalling several important properties of the Lax-Friedrichs numerical flux for Lipschitz continuous Hamiltonian, which have been established in \cite{crandall1984} and are essential for the subsequent analysis.  
\begin{Lemma}\l{lemma:lax}
Let $\bar{f}$ as in \eqref{eq:lax} and $(t,x,u,k)\in \bar{\cQ}_T\t \R^2$, and suppose  Assumption \ref{assum:mono} and the condition $\theta>C(\sup_{\a\in \bA}|\sigma^\a|_0)\lambda$ hold, where $C$ is the Lipschitz constant of the driver $f$ in Assumption \ref{assum:mono}.
\begin{enumerate}[(1)]
\item (Consistency.) For any  test function $\phi\in C^{1,2}(\bar{\cQ}_T)$, we have
\bb\l{eq:f_consistent}
|\bar{f}^{\a}(t,x,u,  \Delta \phi^{n}_{j,i},k)-\bar{f}^{\a}(t,x,u, D\phi(\x^{n}_{j,i}),k)|\le O(h^2/\Delta t).
\ee
\item (Monotonicity.) If $U^n_{i}\ge V^n_{i}$ for all $i,n$, then we have
\bb\l{eq:lax_mono}
\Delta t\bar{f}^{\a}(t,x,u,  \Delta U^{n}_{i},k)+2d\theta U_{i}^n\ge \Delta t\bar{f}^{\a}(t,x,u,  \Delta V^{n}_{i},k)+2d\theta V_{i}^n.
\ee
\item (Stability.) For any bounded functions $U$ and $V$, we have
$$|(\Delta t\bar{f}^{\a}(t,x,u,  \Delta V^{n}_{i},k)+2d\theta V_{i}^n)
-( \Delta t\bar{f}^{\a}(t,x,u,  \Delta U^{n}_{i},k)+2d\theta U_{i}^n)|
\le 2d\theta|U-V|_0.$$
\end{enumerate}

\end{Lemma}

The next proposition presents the monotonicity of the scheme \eqref{eq:hjbi_d}, which plays an important role in the stability and convergence analysis of the discrete equation.  
The proof is an extension of the standard case allowing for a potentially non-Lipschitz, monotone nonlinearity of the driver in $u$ and nonlinearity in the jump term, and we include it for the convenience of the reader.

\begin{Proposition}\l{prop:mono}
Under Assumptions \ref{assum:mono} and \ref{assum:density}, the discrete equation \eqref{eq:hjbi_d} is monotone, i.e., it holds for any functions $U^{n+1}$ and $X_i^n\ge Y_i^n$, $\fa i, n$, that
$$G_h(t_{n+1},x_i,U_i^{n+1},\{X^{b+1}_a\}_{(a,b)\not=(i,n)})\le G_h(t_{n+1},x_i,U_i^{n+1},\{Y^{b+1}_a\}_{(a,b)\not=(i,n)}),$$
provided that 
the following CFL conditions are satisfied:
\bb\l{eq:cfl}
1-\Delta t \Gamma_1(r,\kappa)-2d\theta \ge 0, \q \theta>C(\sup_{\a\in \bA}|\sigma^\a|_0)\lambda,
\ee
where $\Gamma_1$ is defined in \eqref{eq:gamma1},
and  $C$ is the Lipschitz constant of the driver $f$ in Assumption \ref{assum:mono}.
\end{Proposition}
\begin{proof}
Let $X_i^n\ge Y_i^n$, $i\in \Z^d$, $n=0,\ldots, N-1$, we can deduce from  the inequality  $\inf_{\a}S-\inf_{\a}T\le \sup_\a (S-T)$ and  \eqref{eq:Ah} that it suffices to establish for any given $\x_i^n=(t_n,x_i)$ that
\begin{align*}
X^n_i+&\Delta t \sum_{j\not=i}l^{\a,n+1}_{r,h,j,i}X^{n+1}_j+\Delta t K^\a_{r,h}X^n_i+\Delta t\bar{f}^{\a}(\x^n_i,U^{n+1}_i,\Delta X_i^n,B^{\a}_{r,h} X^n_i)\\
&-
\big[Y^n_i+\Delta t \sum_{j\not=i}l^{\a,n+1}_{r,h,j,i} Y^{n+1}_j+ \Delta t K^\a_{r,h}Y^n_i+\Delta t\bar{f}^{\a}(\x^n_i,U^{n+1}_i,\Delta Y_i^n,B^{\a}_{r,h} Y^n_i)\big]\ge 0,
\end{align*}
which by using \eqref{eq:Kh} and the fact that $l^{\a,n+1}_{r,h,j,i},k^{\a,n+1}_{r,h,j,i}\ge 0$ can be reduced to showing
\begin{align}
(1&-\Delta t \sum_{j\not = 0}k^{\a,n+1}_{r,h,j,i}-2d\theta )(X^n_i-Y^n_i)+\Delta t[\bar{f}^{\a}(\x^n_i,U^{n+1}_i,\Delta X_i^n,B^{\a}_{r,h} X^n_i)-\bar{f}^{\a}(\x^n_i,U^{n+1}_i,\Delta X_i^n,B^{\a}_{r,h} Y^n_i)]\nb\\
+&\Delta t[\bar{f}^{\a}(\x^n_i,U^{n+1}_i,\Delta X_i^n,B^{\a}_{r,h} Y^n_i)+2d\theta X_i^n-\bar{f}^{\a}(\x^n_i,U^{n+1}_i,\Delta Y_i^n,B^{\a}_{r,h} Y^n_i)-2d\theta Y_i^n]\ge 0. \l{eq:scheme_mono_1}
\end{align}
Suppose 
the condition $\theta>C(\sup_{\a\in \bA}|\sigma^\a|_0)\lambda$ is satisfied,
we can then deduce from the monotonicity \eqref{eq:lax_mono} and the definition \eqref{eq:lax} of the numerical flux $\bar{f}$ that it remains to obtain a lower bound of 
$f^{\a}(t,x,u,z,B^{\a}_{r,h} X^n_i)-f^{\a}(t,x,u,z,B^{\a}_{r,h} Y^n_i)$.

Since $f$ is non-decreasing in $k$, we shall assume $B^{\a}_h X^n_i\le B^{\a}_h Y^n_i$, otherwise the lower bound is $0$. 
Using the Lipschitz continuity of $f$ on $k$, we obtain for any $\x_i^n=(t_n,x_i)$  that
\begin{align*}
&f^{\a}(t,x,u,z,B^{\a}_{r,h} Y^n_i)-f^{\a}(t,x,u,z,B^{\a}_{r,h} X^n_i)\le C(B^{\a}_{r,h} Y^n_i-B^{\a}_{r,h} X^n_i)\\
=&\ C\int_{|e|>r}\!\bigg[m\big(\sum_{j\not =0}\om_j(\eta^{\a} (\x^n_i,e);h)[Y^n_{i+j}-Y^n_{i}]\big)\!-\!m\big(\sum_{j\not =0}\om_j(\eta^{\a} (\x^n_i,e);h)[X^n_{i+j}-X^n_{i}]\big)\bigg]\gamma(\x^n_i,e)\,\nu(de)\\
\le &\ C\int_{|e|>r}\!\bigg[\sum_{j\not =0}\om_j(\eta^{\a} (\x^n_i,e);h)[(Y^n_{i+j}-X^n_{i+j})-(Y^n_{i}-X^n_{i})]\bigg]^+\gamma(\x^n_i,e)\,\nu(de)\\
\le &\ C\sum_{j\not =0}\int_{|e|>r}\om_j(\eta^{\a} (\x^n_i,e);h)\gamma(\x^n_i,e)\,\nu(de)\bigg[[(Y^n_{i+j}-X^n_{i+j})-(Y^n_{i}-X^n_{i})]\bigg]^+\\
=&\ C\sum_{j\in \cA_{i}}b_{r,h,j,i}^{\a,n}[(X^n_{i}-Y^n_{i})-(X^n_{i+j}-Y^n_{i+j})],
\end{align*}
with   $b_{r,h,j,i}^{\a,n}$ defined as in \eqref{eq:brh}  and the index set $\cA_i\coloneqq\{j\in \Z^d\mid Y^n_{i+j}-X^n_{i+j}> Y^n_{i}-X^n_{i}\}$.
Therefore, the above estimate together with \eqref{eq:scheme_mono_1} implies that the discrete equation \eqref{eq:hjbi_d} is monotone provided that
\bb
1-\Delta t \bigg(\sum_{j\not = 0}k^{\a,n+1}_{r,h,j,i}+C\sum_{j\not = 0}b^{\a,n}_{r,h,j,i}\bigg)-2d\theta \ge 0,
\ee
which along with the estimate \eqref{eq:sumkb} and the fact $\Gamma_1(r,\kappa)>\Gamma_2(r,\kappa)$ for small enough $r$ lead us to the desired CFL condition \eqref{eq:cfl}.
\end{proof}


The following proposition establishes a discrete comparison principle for \eqref{eq:hjbi_d}, which subsequently implies 
the uniqueness of the solution to the discrete equation \eqref{eq:hjbi_d}. 
\begin{Proposition}\l{prop:unique}
Suppose  Assumption \ref{assum:mono}  holds. 
Let $\{X^{n+1}_i\}_i$ and $\{Y^{n+1}_i\}_i$ be two bounded functions such that
$$G_h(t_{n+1},x_i,X_i^{n+1},\{X_a^{n+1},U^{n}\}_{a\not=i})\le G_h(t_{n+1},x_i,Y_i^{n+1},\{Y_a^{n+1},U^{n}\}_{a\not=i}),\q \fa i\in\Z^d,$$
then we have $X^{n+1}_i\le Y^{n+1}_i$ for all $i\in \Z^d$. Consequently, 
the discrete equation \eqref{eq:hjbi_d} admits at most one bounded solution. \end{Proposition}
\begin{proof}
We shall consider the  quantity $m=\sup_i (X^{n+1}_i- Y^{n+1}_i)$.
Since the desired result holds if $m\le 0$, we shall assume $m>0$, which implies  for each small enough $\eps$, there exists an index $i\in \Z^d$   such that
$$
0<m-\eps<X^{n+1}_i-Y^{n+1}_i\le m,
$$
from which we can deduce that $\{X^{n+1}_i\}_i$ and $\{Y^{n+1}_i\}_i$ satisfy the following inequality:
\begin{align*}
&(X_i^{n+1}-Y_i^{n+1})[(X_i^{n+1}-U_i^{n})+\Delta t \inf_{\a\in \bA}\big(-A^{\a}_{r,h} X^{n+1}_i- K^\a_{r,h}U^n_i-\tilde{f}^{\a}(t_n,x_i,X^{n+1}_i,\Delta U_i^n,B^{\a}_{r,h} U^n_i)\big)\\
&-(Y_i^{n+1}-U_i^{n})-\Delta t \inf_{\a\in \bA}\big(-A^{\a}_{r,h} Y^{n+1}_i-K^\a_{r,h}U^n_i-\tilde{f}^{\a}(t_n,x_i,Y^{n+1}_i,\Delta U_i^n,B^{\a}_{r,h} U^n_i)\big)]\le 0,
\end{align*}
which together with the inequality  $-(\inf_{\a}S-\inf_{\a}T)\le -\inf_\a (S-T)$ implies that
\begin{align*}
|X_i^{n+1}-Y_i^{n+1}|^2&\le -\Delta t \inf_{\a\in \bA}\big\{(X_i^{n+1}-Y_i^{n+1})[-(A^{\a}_{r,h} X^{n+1}_i-A^{\a}_{r,h}Y^{n+1}_i)\nb\\
&\hspace{-0.6 cm} -\big(\tilde{f}^{\a}(t_n,x_i,X^{n+1}_i,\Delta U_i^n,B^{\a}_{r,h} U^n_i)-\tilde{f}^{\a}(t_n,x_i,Y^{n+1}_i,\Delta U_i^n,B^{\a}_{r,h} U^n_i)\big)]\big\}\nb\\
&\hspace{-1 cm} \le \Delta t\sup_{\a\in \bA}\big\{-\sum_{m} l^{\a, n+1}_{r,h,m,i}|X_i^{n+1}-Y_i^{n+1}|^2+\sum_{m} l^{\a, n+1}_{r,h,m,i}(X^{n+1}_m-Y^{n+1}_m)(X^{n+1}_i-Y^{n+1}_i)  \nb\\
& \hspace{-0.6 cm} +(X_i^{n+1}-Y_i^{n+1})\big(\tilde{f}^{\a}(t_n,x_i,X^{n+1}_i,\Delta U_i^n,B^{\a}_{r,h} U^n_i)-\tilde{f}^{\a}(t_n,x_i,Y^{n+1}_i,\Delta U_i^n,B^{\a}_{r,h} U^n_i)\big)]\big\}\nb\\
&\hspace{-1 cm} \le \Delta t\sup_{\a\in \bA}\big\{-\sum_{m} l^{\a, n+1}_{r,h,m,i}(m-\eps)^2+\sum_{m} l^{\a, n+1}_{r,h,m,i}m^2\}\nb\le O(\eps),
\end{align*}
and consequently we have $(m-\eps)^2\le O(\eps)$. Letting $\eps\to 0$ yields  $m=0$, which leads  to a contradiction. This completes the proof of the discrete comparison principle for \eqref{eq:hjbi_d}, from which we can directly infer  the uniqueness of bounded solutions to \eqref{eq:hjbi_d}.
\end{proof}

The next result provides an a priori estimate for the solution to \eqref{eq:hjbi_d}. 

\begin{Proposition}\l{prop:apriori}
Suppose  Assumptions \ref{assum:mono} and \ref{assum:density}, and the CFL condition \eqref{eq:cfl} hold. Let $U^{n+1}$ be a bounded solution  to \eqref{eq:hjbi_d},  then we have the following a priori estimate:
\bb\l{eq:apriori}
|U^{n+1}|_0\le \max\{|\zeta(t_n,\cdot)^+|_0, \; |U^n|_0+\Delta t\sup_{\a\in \bA}|f^{\a}(\cdot,\cdot,0,0,0)|_0\}.
\ee
\end{Proposition}
\begin{proof}
Without loss of generality, we can assume $|U^{n+1}|_0>|\zeta(t_n,\cdot)^+|_0$. Then for any small enough $\eps>0$, we can choose an index $i$ such that 
$|U^{n+1}_i|\ge |U^{n+1}|_0-\eps>|\zeta(t_n,\cdot)^+|_0$.

Let us first assume $U^{n+1}_i>0$,  which implies $\rho(\zeta(t_n,x_i)-U^{n+1}_i)^+=0$.
We  then deduce from \eqref{eq:hjbi_d} and $\mu\le 0$ that 
\begin{align}
|U^{n+1}_i|^2=&\ U^{n+1}_i[U^n_i+\Delta t \sup_{\a\in \bA}\big(A^{\a}_{r,h} U^{n+1}_i+K^\a_{r,h} U^n_i+\bar{f}^{\a}(t_n,x_i,U^{n+1}_i,\Delta U_i^n,B^{\a}_{r,h} U^n_i)\big)]
\nb\\
&\hspace{-1 cm}\le \Delta t \sup_{\a\in \bA}\big[U^{n+1}_i\sum_{m}l^{\a, n+1}_{r,h,m,i}(U^{n+1}_m-U^{n+1}_i)+U^{n+1}_i\big(\bar{f}^{\a}(t_n,x_i,U^{n+1}_i,\Delta U_i^n,B^{\a}_{r,h} U^n_i)
\nb\\
&\hspace{-0.6 cm}-\bar{f}^{\a}(t_n,x_i,0,\Delta U_i^n,B^{\a}_{r,h} U^n_i)\big)]
+U^{n+1}_i[U^n_i+\Delta t \sup_{\a\in \bA}\big(K^\a_{r,h} U^n_i+\bar{f}^{\a}(t_n,x_i,0,\Delta U_i^n,B^{\a}_{r,h} U^n_i)\big)]
\nb\\
&\hspace{-1 cm}\le \Delta t \sup_{\a\in \bA}[\sum_{m}l^{\a, n+1}_{r,h,m,i}|U^{n+1}|_0^2-\sum_{m}l^{\a, n+1}_{h,m,i}(|U^{n+1}|_0-\eps)^2]+U^{n+1}_i \sup_{\a\in \bA}\big\{\Delta t\sum_{j\not=0}k^{\a, n}_{r,h,j,i}[U^n_{i+j}-U^n_i]
\nb\\
&\hspace{-0.6 cm}+U^n_i+[\Delta t\bar{f}^{\a}(t_n,x_i,0,\Delta U_i^n,B^{\a}_h U^n_i)\big)+2d\theta U_i^n
-\Delta t\bar{f}^{\a}(t_n,x_i,0,0,B^{\a}_h U^n_i)]-2d\theta U_i^n
\nb\\
&\hspace{-0.6 cm} +\Delta t\bar{f}^{\a}(t_n,x_i,0,0,B^{\a}_h U^n_i)-\Delta t\bar{f}^{\a}(t_n,x_i,0,0,0)+\Delta t\bar{f}^{\a}(t_n,x_i,0,0,0)\big\}. \l{eq:Ubdd_1}
\end{align}
Since the CFL condition \eqref{eq:cfl} is satisfied, we know the numerical flux $\bar{f}$ is  stable in sup-norm, hence we can use the monotonicity of $f$ on $k$ and bound the above expression by
\begin{align}
O(\eps)&+U^{n+1}_i \sup_{\a\in \bA}\big\{(1-2d\theta -\Delta t\sum_{j\not=0}k^{\a, n}_{r,h,j,i})U^n_i+(\Delta t\sum_{j\not=0}k^{\a, n}_{r,h,j,i}+2d\theta)|U^n|_0\nb\\
&+\Delta t C(B^{\a}_{r,h} U^n_i)^++\Delta t {f}^{\a}(t_n,x_i,0,0,0)\big\}. \l{eq:Ubdd_2}
\end{align}
Note by using the properties of $m$, we can obtain a positive upper bound for $B^{\a}_h U^n_i$:
\begin{align*}
B^{\a}_h U^n_i\le C\!\int_{|e|>r} \!\! \bigg(\sum_{j\not=0}\om_j(\eta^{\a} (t_n,x_i,e);h)(U^n_{i+j}\!-\!U^n_{i})\bigg)^+\!\!\gamma(t_n,x_i,e)\,\nu(de)\le C\sum_{j\not=0}b^{\a, n}_{r,h,j,i}(U^n_{i+j}\!-\!U^n_{i})^+,
\end{align*}
from which, along with the index set $\cA_i=\{j\in \Z^d\mid U^n_{i+j}>U^n_{i}\}$, we can further bound  \eqref{eq:Ubdd_2} by:
\begin{align*}
U^{n+1}_i \sup_{\a\in \bA}\big\{[1-2d\theta - & \Delta t\big(\sum_{j\not=0}k^{\a, n}_{r,h,j,i}+C\sum_{j\in \cA_i}b^{\a, n}_{r,h,j,i}\big)]U^n_i+ \\
& \big[\big(\Delta t\sum_{j\not=0}k^{\a, n}_{r,h,j,i}+C\sum_{j\in \cA_i}b^{\a, n}_{r,h,j,i}\big)+2d\theta\big]|U^n|_0
+\Delta t {f}^{\a}(t_n,x_i,0,0,0)\big\}+O(\eps),
\end{align*}
then by using the CFL condition \eqref{eq:cfl} and the estimate \eqref{eq:Ubdd_1},  we obtain for small enough $r$ that 
\begin{align}\l{eq:a prioir_eps}
(|U^{n+1}|_0-\eps)^2\le |U^{n+1}_i|^2\le |U^{n+1}|_0(|U^n|_0+\Delta t\sup_{\a\in \bA}|{f}^{\a}(\cdot,\cdot,0,0,0)|_0)+O(\eps) .
\end{align}
For the case with $U^{n+1}_i<0$, we have $\rho(\zeta(t_n,x_i)-U^{n+1}_i)^+U^{n+1}_i\le 0$, and hence one can derive the same estimate \eqref{eq:a prioir_eps} similarly by considering $-U^{n+1}$. Then
letting $\eps\to 0$ and dividing both sides by $|U^{n+1}|_0$ give us the desired result \eqref{eq:apriori}.
\end{proof}

The next lemma shows that the discrete equation \eqref{eq:hjbi_d} admits a unique bounded solution provided that the driver $f$ is  Lipschitz in $y$, which has been established in \cite{dumitrescu2018} for $m(x)=x$ by reformulating the equation into a contraction mapping on the Banach space of bounded functions on $h\Z^d$ endowed with the sup-norm. 
The extension to general Lipschitz continuous $m$ is straightforward and therefore the proof omitted.

\begin{Lemma}\l{lemma:existence_lipschitz}
Suppose  Assumptions \ref{assum:mono} and \ref{assum:density},  and the CFL condition \eqref{eq:cfl} hold. If we further assume $f$ is globally Lipschitz continuous in $y$, i.e., 
$$
|f^\a(t,x,y,z,k)-f^\a(t,x,y',z,k)|\le C|y-y'|,\q \fa y,y'\in \R,\, (t,x,z,k)\in \bar{\cQ}_T\t\R^d\t\R,
$$
then the discrete equation \eqref{eq:hjbi_d} admits a unique bounded solution.
\end{Lemma}

Now we are ready to demonstrate the existence of solutions to the discrete equation \eqref{eq:hjbi_d} with a general monotone driver. We shall adapt some arguments for monotone backward stochastic difference equations employed in \cite{pu2017}, by approximating \eqref{eq:hjbi_d} with discrete equations with Lipschitz drivers, whose solutions subsequently enable us to construct the solution of \eqref{eq:hjbi_d}.

\begin{Theorem}\l{thm:existence}
Suppose  Assumptions \ref{assum:mono} and \ref{assum:density},  and the CFL condition \eqref{eq:cfl} hold, then the discrete equation \eqref{eq:hjbi_d} admits a unique bounded solution satisfying  the a priori estimate \eqref{eq:apriori}.
\end{Theorem}

\begin{proof}
The uniqueness and the a priori bound have been established in Proposition \ref{prop:unique} and \ref{prop:apriori}, respectively. We now prove the existence of solution $U^{n+1}$ to \eqref{eq:hjbi_d} with a given $U^n$ in two steps.  

\textbf{Step 1: $f^\a(t,x,0,z,k)$ is uniformly bounded for all $(\a, t,x,z,k)\in \bA\t \bar{\cQ}_T\t\R^d\t\R$.}

For a family of mollifiers $\varrho_m:\R\to (0,\infty)$, which are smooth functions supported in $(-\f{1}{m},\f{1}{m})$ with mass one, i.e. $\int_\R \varrho_m(s)\,ds=1$ for all $m\in \N$, we define 
the regularized drivers:
\bb
f_m^\a(t,x,y,z,k)\coloneqq (f^\a(t,x,\cdot,z,k)\ast \varrho_m)(y),
\ee
which are clearly  monotone in $y$ with $\mu\le 0$ and globally Lipschitz continuous in $x,z,k$ with the same Lipschitz constant as $f$. One can deduce from the uniform boundedness of $f^\a(t,x,0,z,k)$ and the growth condition \eqref{eq:f_growth} that $f^\a_m$ is locally Lipschitz continuous in $y$, uniformly with respect to $(\a, t,x,z,k)$. Therefore, for each $m,p\in \N$, by considering the  truncated drivers $
f^\a_{m,p}(t,x,y,z,k)\coloneqq f^\a_m(t,x,\Pi_p(y),z,k)$ with
$$
\Pi_p(s)=\f{\inf(p,|s|)}{|s|}s, \q s\in \R,
$$
we can construct a family of globally Lipschitz continuous and bounded drivers satisfying
$$
y(f^\a_{m,p}(t,x,y,z,k)-f^\a_{m,p}(t,x,0,z,k))=\f{|y|}{\inf(p,|y|)}\Pi_p(y)(f^\a_{m,p}(t,x,y,z,k)-f^\a_{m,p}(t,x,0,z,k))\le 0.
$$
Therefore, using Lemma \ref{lemma:existence_lipschitz} and following the proof of the a priori estimate \eqref{eq:apriori}, we know for each $m,p\in \N$, there exists a unique solution $U^{m,p}$ (where we omit the superscript  $n+1$ for simplicity) solving \eqref{eq:hjbi_d}  with the numerical flux associated to $f^\a_{m,p}(t,x,y,z,k)$ and satisfying the following estimate:
\begin{align*}
|U^{m,p}|_0&\le  \max\{|\zeta(t_n,\cdot)^+|_0, \, |U^n|_0+\Delta t\sup_{\a\in \bA}|f^{\a}_{m,p}(\cdot,\cdot,0,0,0)|_0\}\\
&=  \max\{|\zeta(t_n,\cdot)^+|_0, \, |U^n|_0+\Delta t\sup_{\a\in \bA}|f^{\a}_{m}(\cdot,\cdot,0,0,0)|_0\}\\
&=  \max\{|\zeta(t_n,\cdot)^+|_0, \, |U^n|_0+\Delta t\sup_{(\a,t,x)\in \bA\t \bar{\cQ}_T}\bigg |\int_{[-\f{1}{m},\f{1}{m}]} f^{\a}(t,x,s,0,0)\varrho_m(-s)\,ds\bigg |\}\\
&\le \max\{|\zeta(t_n,\cdot)^+|_0, \; |U^n|_0+\Delta t\sup_{(\a,t,x,y)\in \bA\t \bar{\cQ}_T\t [-1,1]}|f^{\a}(t,x,y,0,0)|\},
\end{align*}
which is independent of $p$ and $m$. In other words, for large enough $p$, the truncation of the driver $f^\a_m$ has no influence on the solution.
Consequently, we can obtain a family of uniformly bounded functions $\{U_i^m\}_i$ 
such that for each $m\in \N$, $\{U^m_i\}_i$ solves the following regularized equation: 
\bb\l{eq:hjbi_smooth}
\inf_{\a\in \bA}\bigg(\f{U^{m}_i-U^n_i}{\Delta t}-A^{\a}_{r,h} U^{m}_i-K^\a_{r,h}U^n_i-\tilde{f}^{\a}_m(t_n,x_i,U^{m}_i,\Delta U_i^n,B^{\a}_{r,h} U^n_i)\bigg)=0, \q i\in \Z^d.
\ee

Now the uniform boundedness of the sequence of functions $\{U^m\}$ enables us to extract a subsequence, which by a slight abuse of notation is still denoted as $\{U^m\}$, such that for each index $i$, $U^m_i$ converges to some value $U_i$ as $m$ tends to $\infty$. This defines a function  $\{U_i\}_i$ in  $\ell^\infty(\Z^d)$ satisfying the following estimate:
\bb\l{eq:limit}
|U|_0\le \max\{|\zeta(t_n,\cdot)^+|_0, \; |U^n|_0+\Delta t\sup_{(\a,t,x,y)\in \bA\t \bar{\cQ}_T\t [-1,1]}|f^{\a}(t,x,y,0,0)|\}.
\ee

The properties of the mollifier $\varrho_m$ and the continuity of $f$ imply that $f_m^\a(t,x,y,z,k)$ converges to $f^\a(t,x,y,z,k)$ as $m$ tends to infinity, uniformly on any compact subset of $\bA\t\cQ_T\t \R\t \R^d\t\R$. Also for each given $i\in \Z^d$, 
we know the number of terms summed in $A^{\a}_{r,h} U^{m}_i$ is finite uniformly in $\a$. Therefore, we can conclude by using the inequality $|\inf_\bA f-\inf_\bA g|\le \sup_\bA|f-g|$ and passing  $m\to\infty$ in \eqref{eq:hjbi_smooth} that $\{U_i\}$ satisfies the discrete equation \eqref{eq:hjbi_d}. 

\textbf{Step 2: The general case.}

We shall approximate the driver $f$ by the following sequence (different from $f_m^\a$ earlier):
$$
f_p^\a(t,x,y,z,k)=f^\a(t,x,y,z,k)-f^\a(t,x,0,z,k)+\Pi_p(f^\a(t,x,0,z,k)),\q p\in \N,
$$
which clearly fulfils all the assumptions of Step 1, and converges locally uniformly to ${f}^\a(t,x,y,z,k)$. Thus for each $p\in \N$, there exists a function $\{U^p_i\}_i$ satisfying the following discrete equation:
\bb\l{eq:hjbi_trun}
\inf_{\a\in \bA}\bigg(\f{U^{p}_i-U^n_i}{\Delta t}-A^{\a}_{r,h} U^{p}_i-K^\a_{r,h}U^n_i-\tilde{f}^{\a}_p(t_n,x_i,U^{p}_i,\Delta U_i^n,B^{\a}_{r,h} U^n_i)\bigg)=0, \q i\in \Z^d.
\ee
Moreover, one can deduce from \eqref{eq:limit} and the uniform boundedness of $f^\a(t,x,0,0,0)$ in Assumption \ref{assum:mono} that  for large enough $p$, the sequence of functions $\{U^p\}$ is uniformly bounded:
\begin{align*}
|U^p|_0&\le\max\{|\zeta(t_n,\cdot)^+|_0, \; |U^n|_0+\Delta t\sup_{(\a,t,x,y)\in \bA\t \bar{\cQ}_T\t [-1,1]}|{f}^{\a}_p(t,x,y,0,0)|\}\\
&=\max\{|\zeta(t_n,\cdot)^+|_0, \;|U^n|_0+\Delta t\sup_{(\a,t,x,y)\in \bA\t \bar{\cQ}_T\t [-1,1]}|{f}^{\a}(t,x,y,0,0)|\}.
\end{align*}
Then similar arguments as those in Step 1 enable us to extract a subsequence, which converges pointwise to a bounded function $U$.  Then we can pass $p\to \infty$ in \eqref{eq:hjbi_trun} and establish that $U$ solves \eqref{eq:hjbi_d}  for each $i\in \Z^d$, which consequently completes our proof for the existence of solutions.
\end{proof} 

The next result concludes the convergence of the discrete equation.
\begin{Theorem}\l{thm:discrete_conv}
Suppose  Assumptions \ref{assum:mono} and \ref{assum:density} hold and the CFL condition \eqref{eq:cfl} is satisfied, then for any fixed $\rho\ge 0$, the solution to the discrete equation \eqref{eq:hjbi_d} converges to the  solution of \eqref{eq:hjb_penal} uniformly on compact sets as $h\to 0$.
\end{Theorem}
\begin{proof}
It is straightforward to derive from the one-step estimate \eqref{eq:apriori} and Gronwall's lemma that the numerical solutions are  bounded uniformly in $r$, $\Delta t$ and $h$. 
Also it follows immediately from \eqref{eq:r_consistent}, \eqref{eq:KB_consistent}, \eqref{eq:A_consistent} and \eqref{eq:f_consistent} that the scheme is  consistent  with \eqref{eq:hjb_penal} as $r,\Delta t,h\to 0$.
Given the monotonicity of  the scheme as ascertained by Proposition \ref{prop:mono}, 
we can conclude the convergence of the numerical solution by adapting the standard arguments in \cite{barles1991} to our current context.
\end{proof}

\begin{Remark}\l{remark:soln_penal}
The convergence analysis in \cite{barles1991} does not assume the penalized equation \eqref{eq:hjb_penal} admits a solution. In fact, it considers the semicontinuous envelopes of the discrete solutions and demonstrates they are  viscosity solutions to  \eqref{eq:hjb_penal}. Therefore, the convergence result in Theorem \ref{thm:discrete_conv} and the a-priori estimate \eqref{eq:apriori}  subsequently provide us with a constructive proof for the existence of  solutions of \eqref{eq:hjb_penal}, which are bounded uniformly in $\rho$. Similar arguments can be carried out to demonstrate that the obstacle problem admits a bounded viscosity solution.
\end{Remark}

\section{Policy iteration  for the discrete equation}\l{sec:policy}

In this section, we  propose an efficient method for solving the discrete problem based on policy iteration. We shall further demonstrate  local superlinear convergence  by interpreting the scheme as a nonsmooth Newton method.
Since in practice one usually truncates the discrete equation \eqref{eq:hjbi_d} by localizing it onto a chosen bounded computational domain, and  specifying the behaviour of the solution outside the  domain,
we shall consider the following finite-dimensional problem:  for any given $u^n\in \R^M$, we aim to find $u\in \R^M$ such that for each $i\in \cI\coloneqq \{1,\ldots, M\}$,
\begin{align}\l{eq:hjb_f}
0&=\cG_h^{n+1}[u]_i=
\inf_{\a\in \bA}\bigg(\f{u_i-u^n_i}{\Delta t}-A^{\a}_{r,h} u_i-K^\a_{r,h} u^n_i-\tilde{f}^{\a}(t_n,x_i,u_i,\Delta u_i^n,B^{\a}_{r,h} u^n_i)\bigg), 
\end{align}
where 
 $\tilde{f}^\a(t,x,y,z,k)= \bar{f}^\a(t,x,y,z,k)+\rho(\zeta(t,x)-y)^+$ with the Lax-Friedrichs numerical flux $\bar{f}$.
For simplicity,  we shall denote  \eqref{eq:hjb_f} as $\cG[u]=0$ in the sequel.

We remark that the (finite-dimensional) discrete operators $A^{\a}_{r,h}$, $K^\a_{r,h}$ and $B^\a_{r,h}$ in \eqref{eq:hjb_f} and the numerical flux $\bar{f}$ are in general different from those in \eqref{eq:hjb_penal}, where the first and last rows of these discrete operators will need to be modified to take the boundary conditions into consideration. However,  without loss of generality, we can assume that
\bn
\item the operators $A^{\a}_{r,h}$, $K^\a_{r,h}$,  $B^\a_{r,h}$ and $\bar{f}^\a$ are continuous in $\a$, and
\item for each $\a\in \bA$, the matrix $-A^{\a}_{r,h}$ admits positive diagonals, nonpositive off-diagonals and nonnegative row sums,
\en
where the second property
is motivated by truncating the monotone form \eqref{eq:Ah} of the operator  $A^{\a}_{r,h}$. 
Then the well-posedness of \eqref{eq:hjb_f} follows from similar arguments as those in Section \ref{sec:scheme_analysis}.

In order to design an efficient  iterative scheme for solving \eqref{eq:hjb_f}, we  need to impose the following regularity conditions on the nonlinear function $f$:



\begin{Assumption}\l{assum:f_semi}
There exists a  function $\p^o_y f:\bA\t \bar{\cQ}_T\t\R \t\R^d\t \R\mapsto \R$, which is bounded above by some constant $\mu\in \R$, and satisfies 
 the following properties:

For any $(t,x,y)\in \bar{\cQ}\t \R$ and compact subset $\cK$ of  $\bA\t \R^d\t\R$, there exists a neighbourhood $\cU_y$ of $y$ and a constant $C_y>0$, such that 
\bn
\item  the family of functions $\{\p^o_y  f(\cdot, t,x,u,\cdot,\cdot)\}_{u\in\cU_y}$ is equicontinuous and uniformly bounded by the constant $C_y$ with respect to $(\a,z,k)\in \cK$;
\item the following identity holds uniformly with respect to $(\a,z,k)\in \cK$:
\bb\l{eq:semismooth}
\lim_{h\to 0}\f{ f^{\a}(t,x,y+h,z,k)- f^{\a}(t,x,y,z,k)-\p^o_y  f(\a,t,x,y+h,z,k)h}{h}=0.
\ee
\en
\end{Assumption}
\begin{Remark}
An immediate consequence of Assumption \ref{assum:f_semi}
and the continuity of the operators on $\a$ is that 
for any given $u^n\in \R^M$, the nonlinear function $\cG[u]$ defined in \eqref{eq:hjb_f} is locally Lipschitz continuous on $\R^M$, which is crucial for our subsequent analysis.
\end{Remark}
It is clear that Assumption \ref{assum:f_semi} is satisfied if $f$ is continuously differentiable and monotone  in $y$. However, unlike \cite{xu2017}, we do not require $\p^o_y f$ to be continuous in $y$ and  $f^\a$ to be Fr\'{e}chet-differentiable in order to enjoy the property \eqref{eq:semismooth}. 
In fact, \eqref{eq:semismooth} is closely related to  slantly differentiable functions introduced in \cite{chen2000,hintermuller2002}, which contain piecewise  differentiable functions, convex functions and more generally semismooth functions. Since most problems arising from finance are related to \eqref{eq:hjb_f} with semismooth drivers (see e.g.~\cite{epstein1992,karoui1994} and Section \ref{sec:ambiguity} for details),
Assumption \ref{assum:f_semi} applies to a wide range of optimal control and stopping problems   that are of our interest. 
For notational consistency, we shall denote  by $\p^o_y f^\a$  the dependence of $\p^o_y f$ on the controls.

Let $u^n$ be a given solution at the previous discrete time point, $\a\in \bA^M$ and $u\in \R^M$. We introduce the  diagonal matrix $P^\a[u]$ with diagonal entries $\p^o_y f^{\a_i}(t_n,x_i,u_i,\sigma_r^{\a_i}(t_n,x_i)^T\Delta u_i^n/2h,B^{\a_i}_{r,h} u^n_i)$, $i\in\cI$,
which are bounded above by $\mu$ as stated in Assumption \ref{assum:f_semi}.
Moreover, to handle the penalty term $\rho(\zeta-u)^+$, for any given $u=(u_1,\ldots, u_M)^T\in \R^M$, we shall introduce the  diagonal matrix $V^+[u]=\{v_{ij}[u]\}$ with $v_{ii}[u]=-\rho 1_{ \{ \zeta(t_n,x_i)-u_i>0\}}$ for each $i\in \cI$.

With these matrices in hand,  we shall introduce the following mapping 
$\cL^{\a}:\R^M\to \R^{M\t M}$ for any $\a\in \bA^{M}$, which maps any given $u\in\R^M$ into a matrix $\cL^{\a}[u]$, whose $i$-th row is defined as: 
\bb\l{eq:cL}
\cL^{\a}[u]_i\coloneqq (I-\Delta tA^{\a_i}_{r,h})_i -\Delta t (P^\a[u]_i + V^+[u]_i),\q i=1,\ldots,M.
\ee

Now we are ready to present our policy iteration algorithm, which extends the classical Howard algorithm \cite{bokanowski2009,witte2011} to the current nonlinear context.

\begin{Algorithm}\l{algorithm:semismooth}
Set $u^{(0)}=u^n$. Given $u^{(k)}$, $k\ge 0$, the next iterate $u^{(k+1)}$ is computed as follows:

\begin{description}
\item[Policy improvement step.] Compute $\a^{(k+1)}=\{\a^{(k+1)}_i\}_{i=1}^M$ such that for each $i\in \cI$, 
\bb\l{eq:improve}
\a^{(k+1)}_i\in \argmin_{\a\in \bA}\cG[u^{(k)}]_i.
\ee
\item[Policy evaluation step.]  Compute  $u^{(k+1)}\in \R^M$ by solving
\bb\l{eq:eval_l}
\cG[u^{(k)}]+\cL^{(k+1)}[u^{(k)}](u^{(k+1)}-u^{(k)})=0,
\ee
where $\cL^{(k+1)}[u^{(k)}]$ is the matrix \eqref{eq:cL} evaluated at the control $\a^{(k+1)}$ and the iterate $u^{(k)}$. 
\end{description} 
\end{Algorithm}

We now proceed to investigate the  convergence of Algorithm \ref{algorithm:semismooth} by regarding it as a Newton's method to the nonlinear function $\cG$, where $\cL^{(k+1)}[u^{(k)}]$ plays the essential role of  the derivative of $\cG$ at the point $u^{(k)}$. Since in general $\cG$ is not Fr\'{e}chet-differentiable, we shall interpret the derivative in the sense of slant differentiability \cite{chen2000}. Recall that given  two Banach spaces $\cX$ and $\cY$,  a  function $F:\cX\to \cY$ is said to be slantly differentiable in an open set $\cU\subset \cX$ if there exists a family of bounded linear operators $\{\cL[u]\}_{u\in \cU}$ from $\cX$ into $\cY$, called a slanting function for $F$ in $\cU$, such that for all $u\in \cU$ we have
$$
\lim_{|h|_\cX\to 0}|h|_\cX^{-1}|F(u+h)-F(u)-\cL[u+h](h)|_\cY=0.
$$

We start with the convergence analysis of the optimal controls. For each vector $u\in\R^M$, 
we write $\bA_{u,i}=\argmin_{\a\in \bA}\cG[u]_i$, $i=\cI$,  
and define $\bA_u\coloneqq \prod_{i=1}^M\bA_{u,i}$ to be the set of minimizers of $\cG[u]$.
The following result demonstrates the convergence of sets of optimal controls in terms of  the Hausdorff metric, which is defined as 
$$
d_\cH(A,B)\coloneqq \max\bigg\{\sup_{x\in A}\inf_{y\in B}d_\bA(x,y),\,\sup_{x\in B}\inf_{y\in A}d_\bA(x,y) \bigg\},
$$
for any given non-empty subsets $A$ and $B$ of the control set $\bA$ endowed with the metric $d_\bA$.

\begin{Proposition}\l{thm:control_conv}
Suppose Assumptions \ref{assum:mono} and \ref{assum:f_semi} hold.  For any given $u, u'\in \R^M$, we have that
$\max_{i\in \cI}d_\cH(\bA_{u',i},\bA_{u,i})\to 0$ as $u'\to u$.
\end{Proposition}

\begin{proof}
Let  $u\in \R^M$ and 
$i=1,\ldots,M$
  be fixed. We first show that  $\sup_{x\in \bA_{u',i}}\inf_{y\in \bA_{u,i}}d_\bA(x,y)\to0 $ as $u'\to u$.
Suppose it does not hold, 
which means there exists $\eps>0$,  $u_n\to u$, and $\a^{u_n}_{i}\in \bA_{u_n,i}$, such that $d_\bA(\a^{u_n}_{i},\bA_{u,i})\ge\eps$ for all $n$,
 then by adapting the arguments for Lemma 3.2 in \cite{bokanowski2009}, we can deduce a contradiction using the local Lipschitz continuity of $\cG$ and the compactness of $\bA$.

It then remains to prove $\sup_{x\in \bA_{u,i}}\inf_{y\in \bA_{u',i}}d_\bA(x,y)\to0 $ as $u'\to u$. Suppose not, then there exists   $\eps>0$,  $u_n\to u$, $\a^{u,n}\in \bA_{u,i}$ and $\a^{u_n}\in  \bA_{u_n,i}$ such that $d_\bA(\a^{u,n},\a^{u_n})\ge \eps$ for all $n$. Note since $u_n\to u$, we can obtain from the first part of this proof that $d_\bA(\a^{u_n},  \bA_{u,i})\to 0$ as $n\to \infty$. Then the triangle inequality
$$
0=\liminf_n d_\bA(\a^{u,n}, \bA_{u,i})\ge \liminf_n \big(d_\bA(\a^{u,n}, \a^{u_n})-d_\bA(\a^{u_n},  \bA_{u,i})\big)\ge \eps,
$$
leads us to a contraction, which enables us to conclude the desired result.
\end{proof}

\begin{Lemma}\l{eq:lemma:cG}
Let $\cG$ as in  \eqref{eq:hjb_f} and suppose Assumptions \ref{assum:mono} and \ref{assum:f_semi} hold. 
Then  $\cG$ is slantly differentiable in $\R^M$ with  slanting function $\cL^{\a}[u]$ defined in \eqref{eq:cL} and $\a\in\bA_u$.

\end{Lemma}

\begin{proof}
Let $u^n$ and $u$ be fixed. Since $(\sigma_r^{\a}(t_n,x_i)^T\Delta u_i^n/2h, B^{\a}_{r,h}u^n)_{\a\in \bA}$ is contained in a compact subset of $\R^d\t \R$, there exists a neighbourhood $\cU$ of $u$ such that \eqref{eq:semismooth} holds for this compact set. For notational brevity, we will denote 
$f^{\a}(\cdot,\cdot,\cdot,\Delta u_i^n,B_{r,h} u^n_i)=f^{\a}(\cdot,\cdot,\cdot,\sigma_r^{\a}(\cdot,\cdot)^T\Delta u_i^n/2h,B^\a_{r,h} u^n_i)$. A similar notation applies to $\p^o_yf^{\a}(\cdot,\cdot,\cdot,\Delta u_i^n,B_{r,h} u^n_i)$.
 
Let $h\in \R^M$ with $u+h\in \cU$. Then for any $\a^u\in \bA_u$, $\a^{u+h}\in \bA_{u+h}$ and $i\in \cI$,
 we have 
\begin{align}
\cG[u]_i=&\ (I-\Delta tA^{\a^{u}_i}_{r,h}) (u)_i-\Delta t\tilde{f}^{\a^{u}_i}(t_n,x_i,u_i,\Delta u_i^n,B^{\a^{u}_i}_{r,h} u^n_i) -u^n_i-\Delta tK^{\a^{u}_i}_{r,h}u^n_i\nb\\
\ge &\ \cG[u+h]_i-\cL^{\a^{u+h}_i}[u+h](h)_i +\Delta t(A^{\a^{u}_i}_{r,h}-A^{\a^{u+h}_i}_{r,h})(h)_i
\nb\\
&+\Delta t\big[f^{\a^{u}_i}(t_n,x_i,u_i+h_i,\Delta u_i^n,B_{r,h} u^n_i)-f^{\a^{u}_i}(t_n,x_i,u_i,\Delta u_i^n,B_{r,h} u^n_i) \nb\\
&-\p^o_y f^{\a^{u}_i}(t_n,x_i,u_i+h_i,\Delta u_i^n,B_{r,h} u^n_i)h_i\big] \nb\\
&+
\Delta t\big[\p^o_y f^{\a^{u}_i}(t_n,x_i,u_i+h_i,\Delta u_i^n,B_{r,h} u^n_i)
-\p^o_y f^{\a^{u+h}_i}(t_n,x_i,u_i+h_i,\Delta u_i^n,B_{r,h} u^n_i)\big] h_i \nb\\
&+\Delta t\rho \big[(\zeta(t_n,x_i)-u_i-h_i)^+-(\zeta(t_n,x_i)-u_i)^++ 1_{\{\zeta(t_n,x_i)-u_i-h_i>0\}} h_i\big]. \l{eq:max2}
\end{align}
Note that \eqref{eq:max2}   vanishes for small enough $h$, so that  we can conclude from the identity \eqref{eq:semismooth}, the equicontinuity of $\{\p_y^o f\}_{u\in \cU}$ in $(\a,z,k)$, uniform continuity of $A_{r,h}^\a$ in $\a$, and Proposition \ref{thm:control_conv} that
\begin{align*}
\cG[u+h]_i-\cG[u]_i-\cL^{\a^{u+h}}[u+h](h)_i\le o(|h|_0).
\end{align*}

On the other hand, we can start with $\cG[u+h]_i$ and deduce  the corresponding lower  bound:
\begin{align}
&\cG[u+h]_i-\cG[u]_i-\cL^{\a^{u+h}_i}[u+h](h)_i\ge o(|h|_0),\q 
\textnormal{as $|h|_0\to 0$,}
\end{align}
which consequently leads to the desired  slant differentiability of $\cG$.
\end{proof}

The following result concludes the local superlinear convergence of Algorithm \ref{algorithm:semismooth}.
\begin{Theorem}
 Suppose Assumptions \ref{assum:mono} and \ref{assum:f_semi} hold.
  Then for all  $\Delta t$ with $1-\mu\Delta t\ge c_0>0$,   $\cL^\a[u]$ is nonsingular and satisfies $|\cL^\a[u]^{-1}|_0\le 1/c_0$. Consequently, the iterates $\{u^{(k)}\}$ generated by Algorithm \ref{algorithm:semismooth} converge superlinearly to 
the solution $u^*$ of \eqref{eq:hjb_f} in a neighbourhood of $u^*$.
\end{Theorem}
\begin{proof}
We can deduce from Assumption \ref{assum:f_semi} and the properties of the matrices $-A^{\a}_{r,h}$ and $-V^+[u]$  that 
for all $\a\in \bA^M$, $u\in \R^M$ and $\Delta t$ with $1-\mu\Delta t>0$,   
$\cL^\a[u]=\{l^{\a,u}_{ij}\}$ defined  in \eqref{eq:cL} is a strictly diagonally dominant matrix whose row sums satisfy: 
$$
\min_{1\le i\le M}\bigg(l^{\a,u}_{ii}-\sum_{j=1,j\not=i}^M|l^{\a,u}_{ij}|\bigg)\ge 1-\Delta t \mu\ge c_0>0, 
$$
from which, along with Theorem A in \cite{varga1976}, we obtain the desired estimate for $|\cL^\a[u]^{-1}|_0$. Then we can directly infer the local superlinear convergence of Algorithm \ref{algorithm:semismooth}   from \cite[Theorem~3.4]{chen2000}.
\end{proof}

We end this section with an important remark about the implementation of the algorithm. Recall that at the policy improvement step \eqref{eq:improve}, one needs to compute the optimal policy at each computational node, which may not  admit any  analytical expression due to complicated nonlinearities of the PDE coefficients or the approximation operators on the control variable.
In these cases, for each $\eps>0$, we can approximate the original coefficients of \eqref{eq:hjb_penal} by suitable functions $\{f^{\a}_\eps,\sigma_\eps^\a, b^\a_\eps,\eta_\eps^\a\}$ which can be easily optimized, such that the following estimate
$$
|f^\a(\cdot,\cdot,\cdot,0,0)-f^{\a}_\eps(\cdot,\cdot,\cdot,0,0)|_{\bar{\cQ}_T\t 
 [-\varphi(|u|)_0,\varphi(|u|)_0]}+|\sigma^\a-\sigma_\eps^\a|_0+|b^\a-b^\a_\eps|_0+|\int_{E} |\eta^\a-\eta_\eps^\a|^2\,\nu (de)\big|^\f{1}{2}_0=O(\eps^2)
$$
holds uniformly in $\a\in \bA$. Then using the continuous dependence result in Theorem \ref{thm:conts}, we can infer that this approximation error is of magnitude $O(\eps)$. 
Commonly used approximating functions can be constructed by discretizing the admissible control set $\bA$, and performing piecewise constant or piecewise linear approximations.

\section{Numerical experiments}\l{sec:numerical}
In this section, we 
demonstrate the effectiveness of the  schemes through numerical experiments. We present two examples, an optimal investment problem with model uncertainty, and a consumption-portfolio allocation problem with non-Lipschitz recursive utilities. Both examples are related to non-standard HJB equations, where the first example contains non-smooth  convex/concave nonlinearities, while the second one involves monotone drivers of polynomial growth.

\subsection{Optimal investment under ambiguity}\l{sec:ambiguity}
We study first an optimal investment problem over a time interval $[0,T]$ in a financial market with a risk-free asset  and a risky asset. 
For our numerical tests, we assume the interest rate is zero, and the price of the risky asset follows the  jump-diffusion process:
\begin{align*}
dS_t=S_{t-}\big(bdt+\sigma dW_t+ (1\wedge |e|) \,\tilde{N}(dt,de)\big),
\end{align*}
where $W$ is a Brownian motion and  $\tilde{N}(dt,de)=N(dt,de)-\nu(de)dt$ is an independent compensated  Poisson process defined on a probability space $(\Om, \{\cF_t\}_{t\in[0,T]}, \bP)$.

An investor with initial wealth $x>0$ at time $t$ can control their wealth process $X^{t,x,\a}$ through a selection of the portion $\a_t$ of  wealth  allocated in the risky asset, and also the duration of the investment via a stopping time $\tau\in [t,T]$, which leads to the following wealth process:
\begin{align*}
dX^{t,x,\a}_s=\a_{s}X^{t,x,\a}_{s-}\big(bds+\sigma dW_s+ (1\wedge |e|) \,\tilde{N}(ds,de)\big), \; s\in[t,\tau];\q X_t=x,
\end{align*}
and  the terminal payoff  $\xi^{t,x,\a}_\tau = g(X_\tau^{t,x,\a})$.

The aim of the agent is to maximize
the expected performance of the investment by taking ambiguity into account in the spirit of \cite{roger2006, karoui2009}. More precisely,  for given parameters $r, R,\kappa_1, \kappa_2>0$, we consider the following value function:
\bb
u_*(t,x)\coloneqq \sup_{\tau\in \cT_t}\sup_{\a\in \cA_t} \cE^t_{\tau,*}[\xi^{t,x,\a}_\tau]=\sup_{\tau\in \cT_t}\sup_{\a\in \cA_t} \inf_{\b\in\cB_t,\mathbb{Q}\in \cM}\ex_\mathbb{Q}\bigg[\exp\big(-\int_{t}^\tau \b_s\,ds\big)\xi^{t,x,\a}_\tau\bigg],
\ee
over all admissible choices of $(\a,\tau)\in \cA_t\t\cT_t $, where 
$\cB_t$ is a class of adapted processes $\b = (\b_s)_{s\in[t,T]}$ valued in $[r, R]$, which represent ambiguous discount rates,
and $\cM$ is a family of absolutely continuous probability measures  with respect to $\bP$ with density
$$ dM^{\pi,\ell}_t=M^{\pi,\ell}_{t-}\bigg(\pi_tdW_t+\int_E \ell_t(e)\,\tilde{N}(de,dt)\bigg); \q M^{\pi,\ell}_0=1,
$$
where $(\pi,\ell)$ are predictable processes satisfying
$|\pi_t|\le \kappa_1 $ and $ 0\le \ell_t(e)\le\kappa_2(1\wedge |e|)$. 
In other words,  the nonlinear expectation  $\cE^{t}_{\tau,*}[\cdot]$  represents the worst-case scenario in a market with uncertainty arising from the discount rate, the Brownian motion, and  the random jump  source (see \cite{roger2006, karoui2009}). Similarly, we  consider the value function associated to the best-case scenario:
\bb
u^*(t,x)\coloneqq \sup_{\tau\in \cT_t}\sup_{\a\in \cA_t} \cE^{t,*}_{\tau}[\xi^{t,x,\a}_\tau]=\sup_{\tau\in \cT_t}\sup_{\a\in \cA_t} \sup_{\b\in\cB_t,\mathbb{Q}\in \cM}\ex_\mathbb{Q}\bigg[\exp\big(-\int_{t}^\tau \b_s\,ds\big)\xi^{t,x,\a}_\tau\bigg].
\ee

Using the dual representation of $\cE^{t,*}_{\tau}[\cdot]$ (resp.~$\cE^{t}_{\tau,*}[\cdot]$),  we can characterize the value function $u^*$ (resp.~$u_*$) as the viscosity solution to the following HJBVI (see   \cite{roger2006, karoui2009, quenez2013,{dumitrescu2016}}): $u(0,x)=g(x)$ for $x\in \R$, and for $(t,x)\in (0,T]\t \R$,
\begin{align}\l{eq:hjbvi_ex1}
\begin{split}
&\min\big\{
u(t,x)-g(x),\inf_{\a\in [0,1]}\big(u_t-L^\a u
-Ru^-+ru^+-\a\kappa_1\sigma |x u_x|-\kappa_2 B^{\a,*} u \big)
\big\}=0,\\
(\textrm{resp.} \, &\min\big\{
u(t,x)-g(x),\inf_{\a\in [0,1]}\big(u_t-L^\a u
-ru^-+Ru^++\a\kappa_1\sigma |x u_x|+\kappa_2 B^\a_* u \big)
\big\}=0,)
\end{split}
\end{align}
where the nonlocal operators $L^\a= A^\a+K^\a$, $B^{\a,*}$ and $B^{\a}_{*}$ satisfy for $\phi\in C^{1,2}([0,T]\t\R)$ that
\begin{align}
\l{eq:k_ex1}
\begin{split}
A^\a\phi(t,x)&= \f{1}{2}\a^2\sigma^2 x^2\phi_{xx}(t,x)+\a b x \phi_x(t,x),
\\
K^\a\phi(t,x)&=\int_{\R\setminus\{0\}}\big(\phi(t,x+\a x\eta(e))-\phi(t,x)-\a x\eta(e) \phi_x(t,x)\big)\,\nu(de),\\
B^{\a,*}\phi(t,x)&=\int_{\R\setminus\{0\}}\big(\phi(t,x+\a x\eta(e))-\phi(t,x)\big)^+ (1\wedge |e|)\,\nu(de),\\
B_*^\a\phi(t,x)&=\int_{\R\setminus\{0\}}\big(\phi(t,x+\a x\eta(e))-\phi(t,x)\big)^- (1\wedge |e|)\,\nu(de).
\end{split}
\end{align}

We now specify the  choice of   parameters for our experiments. For the jump component, we shall consider a symmetric Variance Gamma model (see e.g. \cite{madan1998}) with  a L\'{e}vy measure $\nu(de)={\exp(-\mu|e|)}/{|e|}de$ on $\R$ and intensity $\eta(e)=1\wedge |e|$, while for the initial condition and the obstacle of  the HJBVI we use the exponential
utility function $g(x) =1 -2e^{-2x}$, which implies the solution of \eqref{eq:hjbvi_ex1} changes its sign on the  domain and hence ensures both $u^+$ and $u^-$ in \eqref{eq:hjbvi_ex1} have effects on the solution. 
We consider the value functions at $(T,x_0)$ with the model parameters in Table \ref{table:ex1parameter}.

\pagebreak
\begin{tablehere}
\centering
\begin{tabular}{||c|c|c|c|c|c|c|c|c||}\hline
 $b$ & $\sigma$  & $\mu$ &   $r$ &  $R$ & $\kappa_1$ & $\kappa_2$ &$T$&$x_0$\\ \hline
 0.1 & 0.2&  6&  0.02 & 0.04 &0.2& 0.5& 1 & 1\\\hline
\end{tabular}
\caption{Model parameters for the optimal investment   problem under ambiguity.}
\label{table:ex1parameter}
\end{tablehere}
\bigskip

Now we discuss the implementation details and discretization parameters. The  HJBVIs \eqref{eq:hjbvi_ex1} will be localized to the domain $(0,2)$ with $u(\cdot, x) = g(x)$ for $x\in \R \setminus (0, 2)$. 
Since the singularity of the measure $\nu$ behaves like $\log(r)$, $r>0$, around zero, we can deduce from  the consistency and stability analysis in Section \ref{sec:scheme_analysis} along with the choice of parameters that choosing $r=h$, $\lambda=\Delta t/h=1/5$, and $\theta=1/5$ for the numerical flux will lead us to a consistent and stable scheme. To ensure the monotonicity of the scheme, we discretize the first-order and second-order derivative by  the upwind scheme and the central-difference scheme, respectively, and evaluate the nonlocal operators by the mid-point quadrature formula. We further discretize the control set $\bA=[0,1]$ with a mesh $h_\eps=\f{1}{10}$. and stop the policy iteration, at each timestep if the difference between two consecutive iterates is less than $10^{-10}$.
We remark that on the basis of our experiments, this control discretization mesh  seems to be sufficiently small, since further refinements 
lead to a relative difference less than $10^{-7}$ in the value functions, which is negligible compared to other discretization errors. The effect of the control discretization will be investigated more closely in the next example.

Table \ref{table:ex1_u^*} contains, for different mesh sizes and penalty parameters, the numerical solutions of the value function $u_*$ at the point $(T,x_0)$ and the maximal number of iterations among all time steps. The line $(a)$ clearly indicates the efficiency of our policy iteration scheme, which solves the discrete equation \eqref{eq:hjb_f} at the accuracy $10^{-10}$ with a small number of iterations. Moreover, we can infer from the line $(b)$ that for a fixed penalty parameter $\rho$, the numerical solutions converge monotonically to the exact solution. The  asymptotic magnitude of the approximation error can be deduced from line $(d)$, which is of $O(h)+O(\Delta t)$, and seems to be independent of the size of the penalty parameter $\rho$. We  remark that a similar first-order monotone convergence  can be observed for $u^*$, for which a detailed discussion is omitted.

\bigskip
\begin{tablehere}
\centering
\begin{tabular}{||l|c|c|c|c|c|c||}\hline
 & $h$ & 1/40 & 1/80 & 1/160 & 1/320 &1/640 \\ \hline
$\rho=10^3$  & (a)&    4  &  4  &  4  &  4 & 5    \\
& (b) &  0.7292780  &  0.7292918  &  0.7292987  &  0.7293021 & 0.7293038  \\
& (c) &  & 13.788  & 6.879  & 3.433 & 1.715\\
& (d) &  &   &  2.004  &  2.004  &     2.002\\ \hline
$\rho=16\t 10^3$  & (a)&    4  &  4  &  4  &  5 & 4    \\
& (b) &  0.7293262  &  0.7293271  &  0.7293275  &  0.7293277 & 0.7293278  \\
& (c) &  &  0.8616  & 0.4300 &  0.2146 & 0.1068\\
& (d) &  &   &   2.004 &  2.004  & 2.009    \\ \hline
\end{tabular}
\caption{Numerical solutions of the value function $u_*$ for the optimal investment problem with different mesh sizes and penalty parameters. Shown are: (a) the maximal number of iterations among all time points; (b) the numerical solutions $U_{\rho,h}$ at $(T , x_0 )$; (c) the increments $U_{\rho,h} -U_{\rho,2h}$ (in $10^{-6})$ ; (d) the rate of increments $(U_{\rho,2h} - U_{\rho,4h})/(U_{\rho,h} - U_{\rho,2h})$.}
\label{table:ex1_u^*}
\end{tablehere}
\medskip

We proceed to analyze the impact of computational domains by  performing computations  on the domains $(0,2)$ and $(0,4)$ with $h = 1/640$ and $\rho=64\t 10^3$. It can be observed that this enlarged computational domain has a negligible effect on the numerical solution of value functions (a relative difference of $3.7\cdot 10^{-7}$ for $u^*$ and $2.5\cdot 10^{-12}$ for $u_*$). 
Moreover, the maximal number of iterations remains to be 4 for both $u^*$ and $u_*$, which seems to be independent of the size $M$ of the discrete equation \eqref{eq:hjb_f}.

 Finally  we examine the convergence of value functions  in terms of the penalty parameter $\rho$. Table \ref{table:ex1_penal}
presents the numerical results obtained using the domain $(0,2)$ with a fixed mesh size $h=1/640$ and different penalty parameters.
For both $u^*$ and $u_*$, we can infer from lines (a) and (b)  a monotone convergence of  the numerical solutions, with an approximation error proportional to the reciprocal of the penalty parameter, as asserted in Theorem \ref{thm:mono_rho} and \ref{thm:conv_rho}. Then by performing linear regression of the values in line (a) against the reciprocal of penalty parameters, we can  estimate the constant $C_0$ in \eqref{eq:penalty_conv_smooth} and construct a convergent approximation of the free boundary of \eqref{eq:hjbvi_ex1} as suggested in  \eqref{eq:free_approx}.
Figure \ref{fig:ex1_ctrl} compares
the feedback control strategies for $u^*$ (i.e., the best-case scenario) and $u_*$ (i.e., the worst-case scenario) with $\rho=64\t 10^3$, where the white region represents the sets in which the obstacle is active, and otherwise the colour indicates the value of the optimal control, as presented in the panel on the right. It clearly illustrates that the investor in general behaves more conservatively in the worst-case scenario.

\bigskip
\begin{tablehere}
\centering
\begin{tabular}{||l|c|c|c|c|c||}\hline
  $\rho$& &$10^3$& $4\t10^3$ & $16\t10^3$ & $64\t10^3$ \\ \hline
$u^*$& (a) & 0.75071151 & 0.75071215 & 0.75071231& 0.75071235 \\ 
& (b)  & & 0.639 & 0.159 & 0.040\\
& (c) &  &  & 3.9998& 4.0006 \\ \hline
$u_*$& (a)&   0.72930381 & 0.72932303 & 0.72932783 & 0.72932903      \\ 
& (b) & & 19.215 & 4.802 & 1.201\\
& (c)&  &  & 4.0016& 3.9976 \\ \hline
\end{tabular}
\caption{Numerical results of the value functions $u^*$ and $u_*$ for the optimal investment problem with different penalty parameters. Shown are: (a) the numerical solutions $U_{\rho}$ at $(T , x_0 )$; (b) the increments $U_{\rho} -U_{\rho/4}$ (in $10^{-6}$); (c) the rate of increments $(U_{\rho/4} - U_{\rho/16})/(U_{\rho} - U_{\rho/4})$.}
\label{table:ex1_penal}
\end{tablehere}
\medskip

\medskip
\begin{figurehere}
    \centering
    \includegraphics[width=0.49\columnwidth,height=6cm]{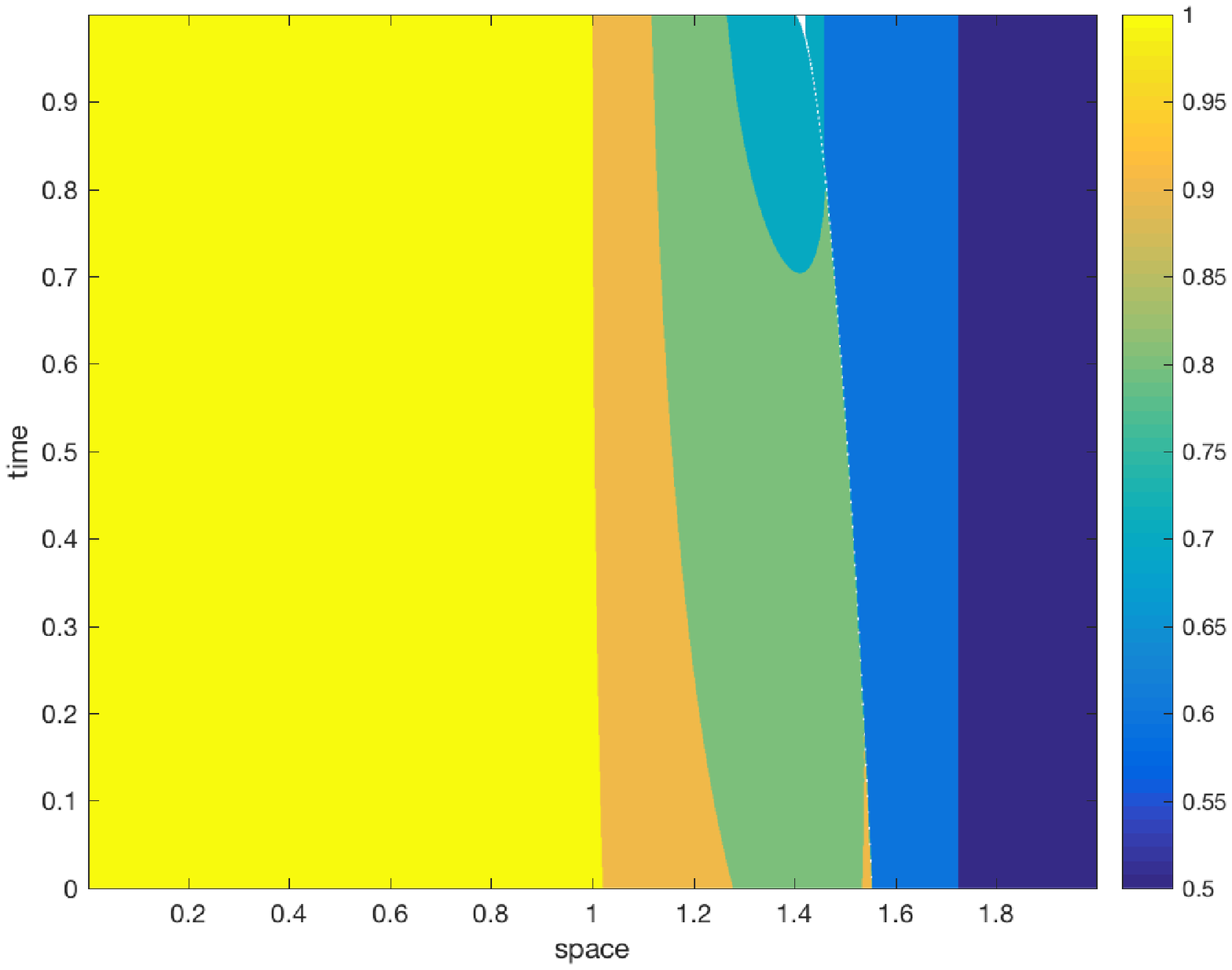}\hfill
    \includegraphics[width=0.49\columnwidth,height=6cm]{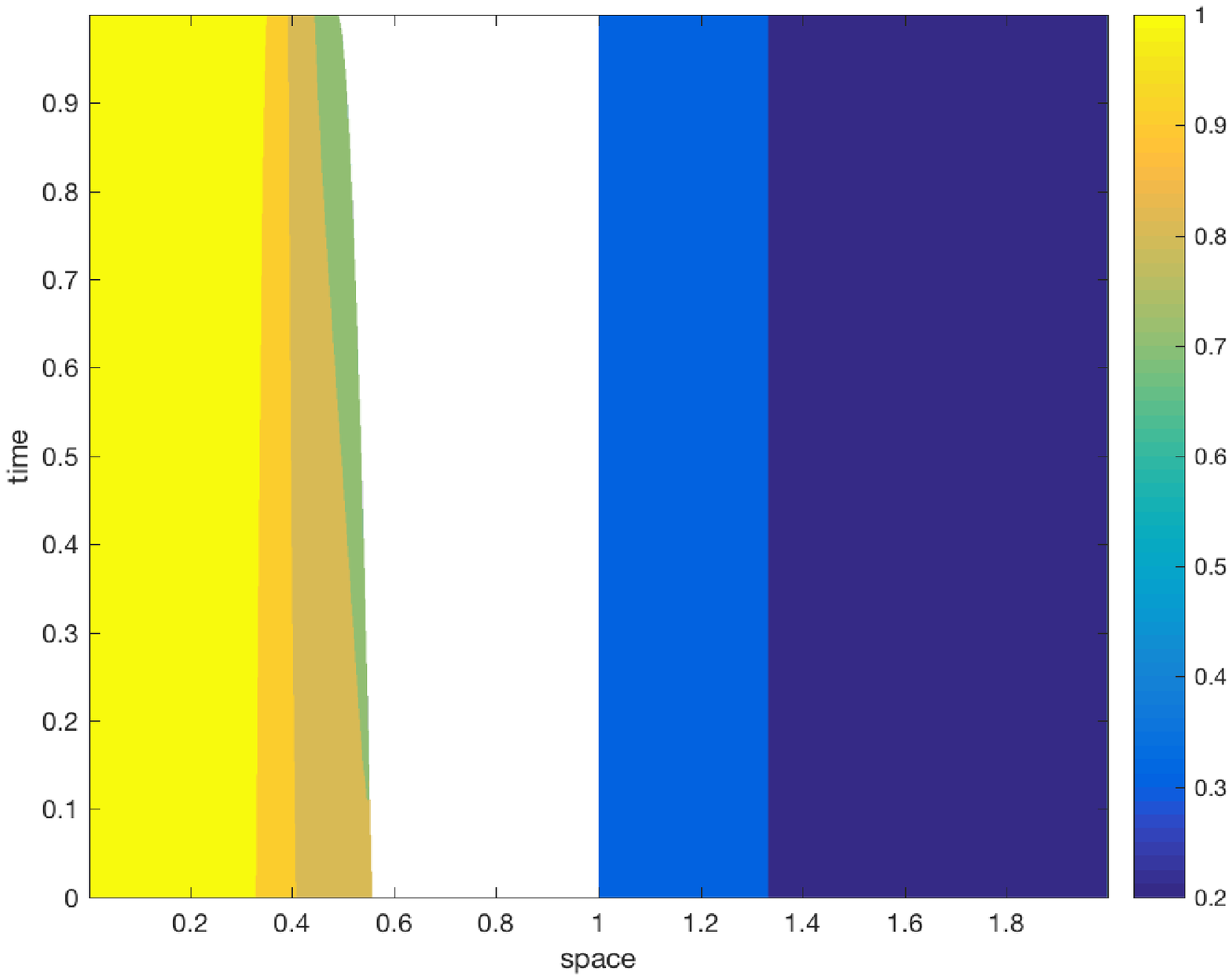}
    \caption{Feedback control strategies with $\rho=16\t 10^3$ for the best-case scenario (left) and the worst-case scenario (right), where the early stopping region is white.}
    \label{fig:ex1_ctrl}
 \end{figurehere}
 \medskip

\subsection{Consumption-portfolio allocation with recursive utility}\l{sec:utility}

As a second example, we shall address a consumption-portfolio maximization problem in terms of recursive utilities, 
which extend the classical additive utilities by allowing one's current well-being to depend on the expected future utilities in a non-risk-neutral way, and play an important role in modern mathematical finance (see e.g.~\cite{kraft2013,pu2017} and references therein). 

For our numerical examples, we shall consider an economy with  a risk-free bond with constant interest rate $r>0$, 
and a risky asset whose dynamics follows a  stochastic volatility model:
\begin{align}
dS_t&=S_t[(r+\lambda v_t)dt+\sqrt{v_t} dW_t],\nb\\
dv_t&=(\vartheta- \kappa v_t)dt+\b\sqrt{v_t}(\rho dW_t+\sqrt{1-\rho^2}d\hat{W}_t),\l{eq:vol}
\end{align}
where $\vartheta, \kappa, \b>0$ are constants, and $W$, $\hat{W}$ are two independent Brownian motions on a filtered probability space $(\Omega,\{\cF_t\}_{t\in[0,T]},\bP)$. 
 
 An agent  controls their wealth by
deciding the portions  invested in stocks and consumed, which implies
the dynamics of the wealth follow the following equation:
$$
dX^{x,\pi,c}_t=X^{x,\pi,c}_t[r+\pi_t\lambda v_t-(1+r)c_t]dt+ \sqrt{v_t} X^{x,\pi,c}_t\pi_t dW_t,\, t\in[0,T]; \q X^{x,\pi,c}_0=x_0,
$$
where $\pi,c:\Omega\t [0,T]\to [0,1]$ are  the  proportion for investment and consumption respectively,  
and $x_0$ is the  initial wealth. 

The preference of the agent between consumption and investment is described by
the well-known (normalized) continuous-time Epstein-Zin utility suggested in \cite{epstein1992}. 
More precisely, suppose the utility from the terminal wealth at the terminal time $T$ is given by $g(X^{x,\pi,c}_T)$,
then the Epstein-Zin recursive utility is defined as
$$
\cE^{c}[g(X^{x,\pi,c}_T)]=Y_t\q \textrm{and} \q Y_t=\ex\bigg[\int_t^Tf(c_sX^{x,\pi,c}_s,Y_s)\,ds+g(X^{x,\pi,c}_T)\mid \cF_t\bigg], \q t\in[0,T],
$$
with the following driver:
\bb\l{eq:epstein}
f(c, y)\coloneqq \f{\delta }{1-\f{1}{\psi}}(1-\gamma)y \bigg[\bigg(\f{c}{((1-\gamma)y)^{\f{1}{1-\gamma}}} \bigg)^{1-\f{1}{\psi}}-1\bigg],
\ee
where we follow the standard parametrization by letting  $\delta>0$ be the rate of time preference and $0< \psi\not= 1$ be the elasticity of intertemporal substitution. 
The objective of the agent is to maximize
the recursive utility over all admissible choices of $\a=(\pi,c)$:
\bb\l{eq:value_ex2}
u(t,x)\coloneqq \sup_{\a\in \cA_t} \cE^{c}[g(X^{x,\pi,c}_T)].
\ee

It has been demonstrated in \cite{kraft2013,pu2017} that for certain empirically important parameters, for instance the coefficients in Table \ref{table:ex2parameter}, which are taken from \cite{kraft2013} and will be used for our numerical test, this driver \eqref{eq:epstein} of the Epstein-Zin utility is non-Lipschitz but monotone in the utility $y$. Moreover, one can identify the value function \eqref{eq:value_ex2} (with a change of time variable) as the  solution to the following HJB equation: $u(0,x,v)=g(x)$ for $(x,v)\in \R^2$, and for $(t,x,v)\in(0,T]\t\R^2$,
\begin{align}\l{eq:hjb_ex2}
\begin{split}
\inf_{(\pi,c)\in \bA}\big(u_t -\f{1}{2}\pi^2x^2vu_{xx}&-\pi \b\rho xvu_{xv}-\f{1}{2}\b^2 vu_{vv}- x[r+\pi\lambda v-(1+r)c]u_x\\
&-(\vartheta-\kappa v)u_v-f(cx,u) \big)=0,
\end{split}
\end{align}
with $\bA=\{(\pi,c)\in [0,1]\t [0,1]\mid \pi+c\le 1\}$.

For the purpose of numerical experiments, we shall take  the negative exponential utility as the initial condition $g(x)=-e^{-x/2}$, and localize the equation on the domain $[0,2]\t[0,0.05]$.
The following homogeneous Neumann  boundary conditions will be imposed as suggested in \cite{ito2009}: 
$$
u_x(t,2,v)=0, \q (t,v)\in [0, T ]\t [0, 0.05];\q u_v(t,x,0.05)=0, \q (t,x)\in [0, T ]\t [0, 2],
$$
while  the equation \eqref{eq:hjb_ex2} itself is set as the boundary condition at $x=0$ and $v=0$.
We remark that based on our experiments with larger computational domains, the error of  the value function caused by this domain truncation appears to be less than $10^{-7}$.

\bigskip
\begin{tablehere}
\centering
\begin{tabular}{||c|c|c|c|c|c|c|c|c|c|c|c||}\hline
 $\gamma$ & $\psi$  & $\delta$ & $r$ &  $\rho$ &  $\lambda$ & $\b$ & $\kappa$ &$\vartheta$&$x_0$&$v_0$& $T$ \\ \hline
 2 & 1.5&  0.08 &  0.05 & -0.5 &0.5 & 0.25 & 5 &0.1125& 1 &0.02 &0.5 \\\hline
\end{tabular}
\caption{Model parameters for the optimal consumption-portfolio allocation problem.} 
\label{table:ex2parameter}
\end{tablehere}
\medskip

The localized HJB equation \eqref{eq:hjb_ex2} is then discretized using the implicit  linear interpolation Semi-Lagrangian scheme (Scheme 2 in \cite{reisinger2017}) with the mesh  size $h_v=h_x=h$  and the time stepsize $\Delta t=4h$, which is monotone and locally first-order accurate. 
We shall further discretize the control set $\bA$ with a mesh $h_\eps$, and for each  time step, terminate policy iteration once the sup-norm of two consecutive iterates is within the threshold  $10^{-6}$.

Table \ref{table:ex2_u} presents the numerical solutions of \eqref{eq:hjb_ex2} at the grid point $(T,x_0,v_0)$ with different spatial mesh size $h$ and a fixed control discretization mesh $h_\eps=1/20$. We can observe from line (a) that our algorithm requires a small number of iterations to obtain an accurate  solution. 
Moreover, lines (b) and (d) indicate the numerical solution converge monotonically with the convergence rate $O(h)+O(\Delta t)$, as the mesh size tends to zero.

\bigskip
\begin{tablehere}
\centering
\begin{tabular}{||c|c|c|c|c|c||}\hline
  $h$  & 1/100 & 1/200 & 1/400 & 1/800 &1/1600 \\ \hline
(a)       &  3  &  4  &  4  &  3 &  3 \\ 
(b)       &  -0.6604205  &  -0.6581355  &  -0.6580512  &  -0.6580101 &    -0.6579897
 \\ 
   (c)  &   & 2.2851 &  0.0843  &  0.0411  & 0.0204 \\

    (d)  &   &  &  27.120  &  2.052  & 2.014
 \\\hline
\end{tabular}
\caption{
Numerical solutions  for the consumption-portfolio allocation problem with different mesh sizes. Shown are: (a) the maximal number of iterations among all time points; (b) the numerical solutions $U_{h}$ at $(T , x_0,v_0)$; (c) the increments $U_{h} -U_{2h}$ (in $10^{-3})$ ; (d) the rate of increments $(U_{2h} - U_{4h})/(U_{h} - U_{2h})$.}
\label{table:ex2_u}
\end{tablehere}
\medskip

We then investigate the  effect of the control discretization by performing computations with a fixed mesh size $h=1/800$ and  different control meshes.
Numerical results are given in Table \ref{table:ex2_ctrl},
from which we can  observe that
the control discretization error decreases rapidly as the meshsize tends to zero,
 and the control mesh $h_\eps=1/20$  already leads to an accurate approximation with a negligible control discretization error.
We further present the optimal investment and consumption allocation corresponding to \eqref{eq:value_ex2} at $t=0$ in Figure  \ref{fig:ex2_ctrl},  where the colour indicates the value of the optimal  feedback control, as shown in the panel on the right. It depicts that the optimal stock allocation in general decreases with respect to the initial wealth, but less sensitive than consumption as observed in \cite{kraft2013}. Moreover, the consumption is insensitive to the volatility, while the investment allocation depends explicitly on the initial states of volatility and wealth.

\bigskip
\begin{tablehere}
\centering
\begin{tabular}{||c|c|c|c|c||}\hline
$h_\eps$ & 1/5 & 1/10 & 1/20 & 1/40 \\ \hline

 (a) &  -0.668535135  &  -0.660102239  &  -0.658010097  &  -0.658005963
 \\ 
(b) &   & 8.4329 &  2.0921  &    0.0041
 \\\hline
\end{tabular}
\caption{Numerical solutions  for the consumption-portfolio allocation problem with different control refinements. Shown are: (a) the numerical solutions $U_{h_\eps}$ at $(T , x_0,v_0)$; (b) the increments $U_{h_\eps} -U_{2h_\eps}$ (in $10^{-3})$.} 
\label{table:ex2_ctrl}
\end{tablehere}
\medskip

\medskip
\begin{figurehere}
    \centering
    \includegraphics[width=0.49\columnwidth,height=6cm]{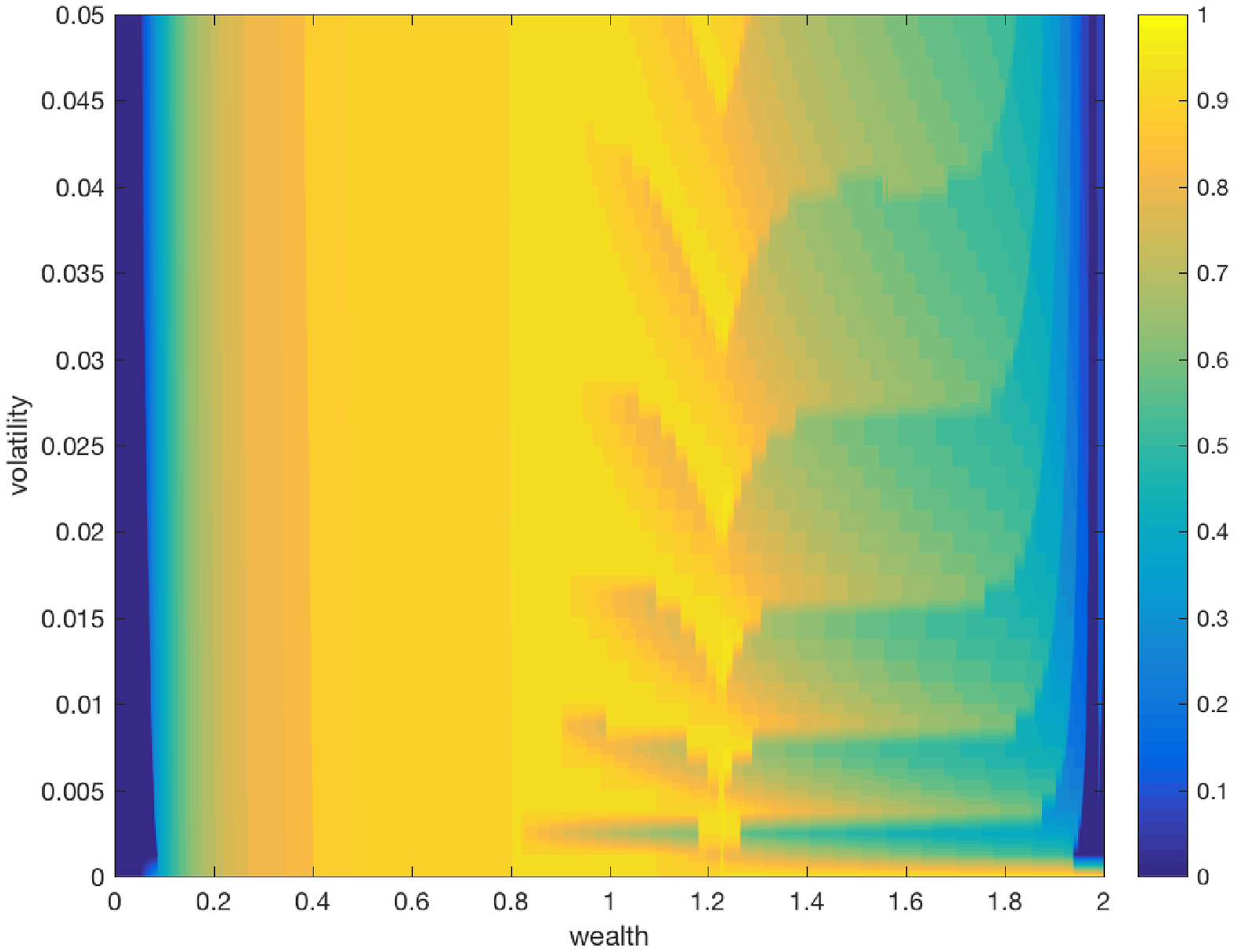}\hfill
    \includegraphics[width=0.49\columnwidth,height=6cm]{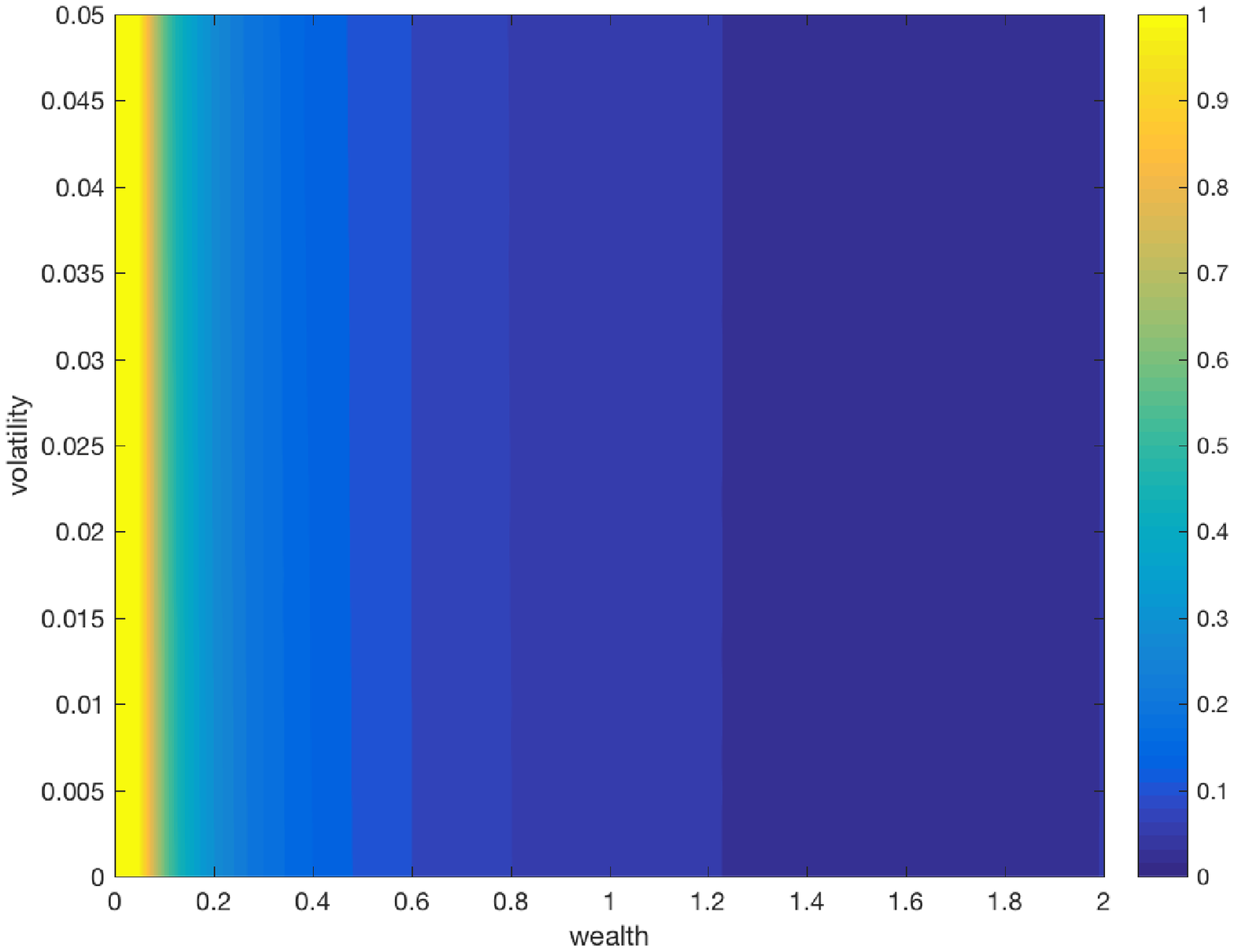}
    \caption{Optimal proportions of wealth for investment  (left) and consumption (right).}
    \label{fig:ex2_ctrl}
 \end{figurehere}
 \medskip

\section{Conclusions}
This paper constructs numerical approximations to the solution and free boundary of HJB variational inequalities with monotone drivers based on the penalty method, monotone schemes, and policy iteration. We prove the convergence of the numerical scheme and illustrate the theoretical results with some numerical examples including an optimal investment under ambiguity  problem and a recursive consumption-portfolio allocation problem.

To the best of our knowledge, this is the first paper which proposes numerical approximations for a HJBVI with a general monotone driver. Natural next steps would be to establish the theoretical convergence rate of  the discretization schemes and to extend this approach to ``double-obstacle" HJBVIs   obtained in \cite{dumitrescu2016mixed} and Hamilton-Jacobi-Bellman-Isaacs equations in \cite{biswas2017}.

\newpage
\appendix
\section{Continuous dependence estimate for penalized equations}\l{sec:appendix}
In this section, we establish continuous dependence estimate for the solutions of the penalized equation \eqref{eq:hjb_penal}, cf.~Theorem \ref{thm:conts}.
\begin{proof}[Proof of Theorem \ref{thm:conts}]
For any given  $\lambda$, $\theta$ and $\eps>0$, we define the following functions:
\begin{align*}
\phi(t,x,y)=\theta e^{\lambda t} |x-y|^2+\eps e^{\lambda t} (|x|^2+|y|^2),\q
\psi(t,x,y)=u_1(t,x)-u_2(t,y)-\phi(t,x,y),
\end{align*}
for all $(t,x,y)\in [0,T]\t \R^d\t \R^d$ and introduce the following quantities:
$$
m^0_{\theta,\eps}=\sup_{\R^d\t \R^d}\psi(0,x,y)^+, \q m_{\theta,\eps}=\sup_{[0,T]\t \R^d\t \R^d}\psi(t,x,y)-m^0_{\theta,\eps}.
$$
The boundedness and semicontinuity  of $u_1$, $u_2$, along with the penalization terms, imply that there exists $(t_0,x_0,y_0)\in [0,T]\t \R^d\t \R^d$, depending on $\theta$ and $\eps$, such that 
$$\psi(t_0,x_0,y_0)=\sup_{[0,T]\t\R^d\t \R^d}\psi(t,x,y)=m_{\theta,\eps}+m^0_{\theta,\eps}.$$
We further introduce the several nonlinear operators, which are essential for the subsequent analysis.
For any given $\a\in \bA$, $\kappa\in (0,1)$, $\phi\in C^{1,2}(\bar{\cQ}_T)$ and bounded semicontinuous function $u$, we define for $i=1,2$ and $\x=(t,x)\in \cQ_T$ that
\begin{align}
K^\a_{\kappa,i}[\phi](\x)&=\int_{|e|\le \kappa}\big(\phi(t,x+\eta^{\a}_i (\x,e))-\phi(\x)-\eta^{\a}_i (\x,e)\cdot \nabla_x\phi(\x)\big)\,\nu_i(de),\\
\tilde{K}^\a_{\kappa,i}[u,p](\x)&=\int_{|e|> \kappa}\big(u(t,x+\eta^{\a}_i (\x,e))-u(\x)-\eta^{\a}_i (\x,e)\cdot p\big)\,\nu_i(de),\\
B^\a_{\kappa,i}[\phi](\x)&=\int_{|e|\le \kappa}m\big(\phi(t,x+\eta^{\a}_i (\x,e))-\phi(\x)\big)\gamma(\x,e)\,\nu_i(de),\\
\tilde{B}^\a_{\kappa,i}[u](\x)&=\int_{|e|> \kappa}m\big(u(t,x+\eta^{\a}_i (\x,e))-u(\x)\big)\gamma(\x,e)\,\nu_i(de).
\end{align}

We  then focus on deriving an  upper bound of $m_{\theta,\eps}$ by first assuming  $m_{\theta,\eps}>0$. This further implies that $t_0>0$ since otherwise we have $m_{\theta,\eps}=\sup_{\R^d\t \R^d}\psi(0,x,y)-m^0_{\theta,\eps}\le 0$. We shall also assume $\mu<0$ in Assumption \ref{assum:mono} (see Remark \ref{remark:strict_monotone}).

Now applying the nonlocal version of Jensen Ishii's lemma \cite[Theorem ~2.2]{jakobsen2005} and using the fact $\inf_\a(a)-\inf_\a(b)\ge \inf_\a(a-b)$, we obtain that for each $\kappa\in (0,1)$, there exist two symmetric matrices $X ,Y\in \R^{d\t d}$ satisfying
\begin{align}
\begin{pmatrix} X & 0\\ 0 &-Y \end{pmatrix} \le 
2\theta e^{\lambda t_0} \begin{pmatrix} I & -I\\ -I &I \end{pmatrix} 
+2\eps e^{\lambda t_0}\begin{pmatrix} I & 0\\ 0 &I \end{pmatrix},
\end{align}
such that the following inequality holds:
\begin{align}
&\lambda \theta e^{\lambda t_0} |x_0-y_0|^2+\lambda \eps e^{\lambda t_0} (|x_0|^2+|y_0|^2)+\inf_{\a\in \bA}\big[\tr(-\sigma^\a_1(t_0,x_0)(\sigma_1^\a (t_0,x_0))^TX 
+\sigma^\a_2(t_0,y_0)(\sigma_2^\a (t_0,y_0))^TY)
\nb \\
&-b^\a_1(t_0,x_0)\nabla_x\phi(t_0,x_0,y_0)-b^\a_2(t_0,y_0)\nabla_y\phi(t_0,x_0,y_0)-l^\a_{K,1}(t_0,x_0)+l^\a_{K,2}(t_0,y_0)\nb\\
&-f_1^\a(t_0,x_0,u_1(t_0,x_0),\sigma^\a_1(t_0,x_0)^T \nabla_x\phi(t_0,x_0,y_0),l^\a_{B,1}(t_0,x_0))\nb\\
&+f_2^\a(t_0,y_0,u_2(t_0,x_0),\sigma^\a_2(t_0,y_0)^T (-\nabla_y\phi(t_0,x_0,y_0)),l^\a_{B,2}(t_0,y_0))\big] \nb\\
&-\rho((\zeta_1(t_0,x_0)-\zeta_2(t_0,y_0))-(u_1(t_0,x_0)-u_2(t_0,y_0)))^+\le 0, \l{eq:ishii}
\end{align}
where the nonlocal terms are defined as:
\begin{align*}
l^\a_{K,1}(t_0,x_0)&\coloneqq K^\a_{\kappa,1}[\phi(\cdot,\cdot,y_0)](t_0,x_0)+\tilde{K}^\a_{\kappa,1}[u_1,\nabla_x \phi(\cdot,\cdot,y_0)](t_0,x_0),\\
l^\a_{K,2}(t_0,y_0)&\coloneqq K^\a_{\kappa,2}[-\phi(\cdot,x_0,\cdot)](t_0,y_0)+\tilde{K}^\a_{\kappa,2}[u_2,-\nabla_y \phi(\cdot,x_0,\cdot)](t_0,y_0),\\
l^\a_{B,1}(t_0,x_0)&\coloneqq B^\a_{\kappa,1}[\phi(\cdot,\cdot,y_0)](t_0,x_0)+\tilde{B}^\a_{\kappa,1}[u_1](t_0,x_0),\\
l^\a_{B,2}(t_0,y_0)&\coloneqq B^\a_{\kappa,2}[-\phi(\cdot,x_0,\cdot)](t_0,y_0)+\tilde{B}^\a_{\kappa,2}[u_2](t_0,y_0).
\end{align*}
In the case that $(\zeta_1(t_0,x_0)-\zeta_2(t_0,y_0))-(u_1(t_0,x_0)-u_2(t_0,y_0))\ge 0$, we can deduce from Lipschitz continuity of $\zeta_1$ and $\zeta_2$ that
\bb\l{eq:obstacle_conts}
m_{\theta,\eps}+m_{\theta,\eps}^0\le u_1(t_0,x_0)-u_2(t_0,y_0)\le |\zeta_1-\zeta_2|_0+C|x_0-y_0|.
\ee
Therefore, in the sequel, we shall  assume  the last term in \eqref{eq:ishii} is equal to $0$.

A straightforward computation gives us that
$$
\nabla_x\phi(t_0,x_0,y_0)=2\theta e^{\lambda t_0} (x_0-y_0)+2\eps e^{\lambda t_0} x_0,\q -\nabla_y\phi(t_0,x_0,y_0)=2\theta e^{\lambda t_0} (x_0-y_0)-2\eps e^{\lambda t_0} y_0,
$$
from which,  together with the Lipschitz continuity of the coefficients, we can deduce the following estimates for the local terms:
\begin{align}\l{eq:local_conts}
\begin{split}
|\tr(-\sigma^\a_1&(t_0,x_0)(\sigma_1^\a (t_0,x_0))^TX 
+\sigma^\a_2(t_0,y_0)(\sigma_2^\a (t_0,y_0))^TY)|\\
\le& 4\theta e^{\lambda t_0}(|\sigma_1^\a-\sigma_2^\a|_0^2+C|x_0-y_0|^2)+2\eps e^{\lambda t_0} C(1+|x_0|^2+|y_0|^2),\\
|b^\a_1(t_0,x_0)&\nabla_x\phi(t_0,x_0,y_0)+b^\a_2(t_0,y_0)\nabla_y\phi(t_0,x_0,y_0)|\\
+&|\sigma^\a_1(t_0,x_0)^T \nabla_x\phi(t_0,x_0,y_0)-\sigma^\a_2(t_0,y_0)^T (-\nabla_y\phi(t_0,x_0,y_0))|\\
\le& 4\theta e^{\lambda t_0}|x_0-y_0|^2+\theta e^{\lambda t_0}(|b^\a_1-b^\a_2|_0^2+|\sigma_1^\a-\sigma_2^\a|_0^2)+\eps e^{\lambda t_0}C(1+|x_0|^2+|y_0|^2).
\end{split}
\end{align}
Moreover, following the same arguments as those for Theorem 4.1 in \cite{jakobsen2005}, we  derive that   
\begin{align}
l^\a_{K,1}&(t_0,x_0)-l^\a_{K,2}(t_0,y_0)\le O(\kappa)+
2\theta e^{\lambda t_0}\bigg[\big|\int_{E} |\eta^\a_1-\eta^\a_2|^2\,(\nu_1\vee\nu_2) (de)\big|_0\nb\\ 
&+\big|\int_{E}(|\eta^\a_1|^2\vee |\eta^\a_2|^2) \,|\nu_1-\nu_2|(de)\big|_0\bigg]
+C\theta e^{\lambda t_0}|x_0-y_0|^2+\eps e^{\lambda t_0}C(1+|x_0|^2+|y_0|^2).\l{eq:K_conts}
\end{align}

We then proceed to estimate the nonlinear terms $f^\a_1$ and $f^\a_2$. For  notational convenience, we denote  $\sigma^\a_1(t_0,x_0)^T \nabla_x\phi(t_0,x_0,y_0)$ and $\sigma^\a_2(t_0,y_0)^T (-\nabla_y\phi(t_0,x_0,y_0))$ by $p_1$ and $p_2$, respectively, and deduce that
\begin{align}
&f_1^\a(t_0,x_0,u_1(t_0,x_0),p_1,l^\a_{B,1}(t_0,x_0))
-f_2^\a(t_0,y_0,u_2(t_0,y_0),p_2,l^\a_{B,2}(t_0,y_0)) \nb\\
\le &f_1^\a(t_0,x_0,u_1(t_0,x_0),p_1,l^\a_{B,1}(t_0,x_0))
-f_1^\a(t_0,x_0,u_1(t_0,x_0),p_1,0) \nb\\
&+f_1^\a(t_0,x_0,u_1(t_0,x_0),p_1,0)
-f_1^\a(t_0,x_0,u_2(t_0,y_0),p_1,0)\nb\\
&+ f_1^\a(t_0,x_0,u_2(t_0,y_0),p_1,0)
-f_1^\a(t_0,y_0,u_2(t_0,y_0),p_2,0)\nb\\
&+ f_1^\a(t_0,y_0,u_2(t_0,y_0),p_2,0)
-f_2^\a(t_0,y_0,u_2(t_0,y_0),p_2,0)\nb\\
&+f_2^\a(t_0,y_0,u_2(t_0,y_0),p_2,0)
-f_2^\a(t_0,y_0,u_2(t_0,y_0),p_2,l^\a_{B,2}(t_0,y_0))\nb
\\
\le& C(l^\a_{B_1})^+ +\mu (m_{\theta,\eps}+m_{\theta,\eps}^0)+C\big[|x_0-y_0|+|\sigma^\a_1(t_0,x_0)^T \nabla_x\phi(t_0,x_0,y_0)-\sigma^\a_2(t_0,y_0)^T (-\nabla_y\phi(t_0,x_0,y_0))|\big] \nb\\
&+ C|\sigma^\a_2(t_0,y_0)^T (-\nabla_y\phi(t_0,x_0,y_0))|+|f_1^\a(\cdot,\cdot,\cdot,0,0)-f_2^\a(\cdot,\cdot,\cdot,0,0)|_{\bar{\cQ}_T\t [-\varphi(|u_2|_0),\varphi(|u_2|_0)]} +C(l^\a_{B_2})^-, \nb
\end{align}
where we have used the fact 
$$
u_1(t_0,x_0)-u_2(t_0,y_0)= \phi(t_0,x_0,y_0)+m_{\theta,\eps}+m_{\theta,\eps}^0\ge m_{\theta,\eps}+m_{\theta,\eps}^0\ge 0,
$$
and the monotonicity of $f$ in $u$. It follows directly from the boundedness of $\sigma^\a$ that 
\bb\l{eq:p2bdd}
|\sigma^\a_2(t_0,y_0)^T (-\nabla_y\phi(t_0,x_0,y_0))|\le 2\theta e^{\lambda t_0}C(|x_0-y_0|^2+1)
+\eps e^{\lambda t_0}C(1+|x_0|^2+|y_0|^2).
\ee

It now remains to bound $(l^\a_{B_1})^+$ and $(l^\a_{B_2})^-$. 
One can obtain from the integrability of the singular measures that  
\begin{align}
|B^\a_{\kappa,1}[\phi(\cdot,\cdot,y_0)](t_0,x_0)|+|B^\a_{\kappa,2}[-\phi(\cdot,x_0,\cdot)](t_0,y_0)|\le O(\kappa).\nb
\end{align}
Moreover, since $\psi$ attains its maximum at $(t_0,x_0,y_0)$, we can deduce 
by using 
$$\psi(t_0,x_0,y_0)\ge \psi(t_0,x_0+\eta^{\a}_1 (t_0,x_0,e),y_0),\q  \psi(t_0,x_0,y_0)\ge \psi(t_0,x_0,y_0+\eta^{\a}_2 (t_0,y_0,e)),
$$
and the property  $-C(-x)^+\le m(x)\le Cx^+$ of the function $m$ that
\begin{align*}
m\big(u_1(t_0,x_0+\eta^{\a}_1 (t_0,x_0,e))-u_1(t_0,x_0)\big)&\le 
\theta  C e^{\lambda t_0}(|x_0-y_0||\eta_1^\a|+|\eta_1^\a|^2)+\eps  C e^{\lambda t_0}(|x_0||\eta_1^\a|+|\eta_1^\a|^2),\\
m\big(u_2(t_0,y_0+\eta^{\a}_2 (t_0,y_0,e))-u_2(t_0,y_0)\big)&\ge 
  -\theta C e^{\lambda t_0}(|x_0-y_0||\eta_2^\a|+|\eta_2^\a|^2)-\eps  C e^{\lambda t_0}(|y_0||\eta_2^\a|+|\eta_2^\a|^2),
\end{align*}
which implies that 
\begin{align*}
\tilde{B}^\a_{\kappa,1}[u_1](t_0,x_0)&=\int_{|e|> \kappa}m\big(u_1(t_0,x_0+\eta^{\a}_1 (t_0,x_0,e))-u_1(t_0,x_0)\big)\gamma(t_0,x_0,e)\,\nu_1(de)\\
&\le  \theta e^{\lambda t_0}C(|x_0-y_0|^2+1)
+\eps e^{\lambda t_0}C(1+|x_0|^2),\\
\tilde{B}^\a_{\kappa,2}[u_1](t_0,y_0)&=\int_{|e|> \kappa}m\big(u_2(t_0,y_0+\eta^{\a}_2 (t_0,y_0,e))-u_2(t_0,y_0)\big)\gamma(t_0,y_0,e)\,\nu_2(de)\\
&\ge - \theta e^{\lambda t_0}C(|x_0-y_0|^2+1)
-\eps e^{\lambda t_0}C(1+|y_0|^2).
\end{align*} 
Consequently, we can bound $C(l^\a_{B_1})^++C(l^\a_{B_2})^-$ by \eqref{eq:p2bdd} with an extra term $O(\kappa)$.

Now we are ready to derive the  upper bound of $m_{\theta,\eps}$. Substituting  
 \eqref{eq:obstacle_conts}, \eqref{eq:local_conts}, \eqref{eq:K_conts}
 and the above estimate of $f^\a_1-f^\a_2$  into \eqref{eq:ishii}, we can obtain that
\begin{align*}
&\lambda \theta e^{\lambda t_0} |x_0-y_0|^2+\lambda \eps e^{\lambda t_0} (|x_0|^2+|y_0|^2)\\
\le & C\theta e^{\lambda t_0}\sup_{\a\in\bA}\big[|\sigma_1^\a-\sigma_2^\a|_0^2+|b^\a_1-b^\a_2|_0^2+\big|\int_{E} |\eta^\a_1-\eta^\a_2|^2\,(\nu_1\vee\nu_2) (de)\big|_0+\big|\int_{E}(|\eta^\a_1|^2\vee |\eta^\a_2|^2) \,|\nu_1-\nu_2|(de)\big|_0\big]\\
&+\sup_{\a\in\bA}|f_1^\a(\cdot,\cdot,\cdot,0,0)-f_2^\a(\cdot,\cdot,\cdot,0,0)|_{\bar{\cQ}_T\t [-\varphi(|u_2|_0),\varphi(|u_2|_0)]}+\mu(  m_{\theta,\eps}+m_{\theta,\eps}^0)+ |\zeta_1-\zeta_2|_0+C|x_0-y_0|\\
&+C\theta e^{\lambda t_0}(|x_0-y_0|^2+1)+\eps e^{\lambda t_0}C(1+|x_0|^2+|y_0|^2)+O(\kappa),
\end{align*}
from some constant $C$ depends only on the coefficients. Then letting $\kappa\to 0$, taking $\lambda=C+1$ and maximizing over $|x_0-y_0|$, we have
\begin{align*}
&-\mu(  m_{\theta,\eps}+m_{\theta,\eps}^0)\le \f{C}{\theta}+C\theta +\sup_{\a\in\bA}|f_1^\a(\cdot,\cdot,\cdot,0,0)-f_2^\a(\cdot,\cdot,\cdot,0,0)|_{\bar{\cQ}_T\t [-\varphi(|u_2|_0),\varphi(|u_2|_0)]}+ |\zeta_1-\zeta_2|_0\\
+ & C\theta \sup_{\a\in\bA}\big[|\sigma_1^\a-\sigma_2^\a|_0^2+|b^\a_1-b^\a_2|_0^2+\big|\int_{E} |\eta^\a_1-\eta^\a_2|^2\,(\nu_1\vee\nu_2) (de)\big|_0+\big|\int_{E}(|\eta^\a_1|^2\vee |\eta^\a_2|^2) \,|\nu_1-\nu_2|(de)\big|_0\big].
\end{align*}

We remark that the above estimate is based on the assumption that $m_{\theta,\eps}>0$. In case that $m_{\theta,\eps}\le 0$, we  obtain from the definition of $m_{\theta,\eps}^0$, the inequality $(x+y)^+\le x^++y^+$ and the Lipchitz continuity of initial conditions that
$$
m_{\theta,\eps}+m_{\theta,\eps}^0\le \sup_{x,y}[|(u_1-u_2)^+|_0+(C|x-y|-\theta |x-y|^2)^+]\le |(u_1-u_2)^+|_0+C/\theta.
$$
Therefore for any $(t,x)\in \bar{\cQ}_T$ and $\eps,\theta>0$, we can deduce from the definition of $m_{\theta,\eps}$ that
\begin{align*}
 &u_1(t,x)-u_2(t,x)\le m_{\theta,\eps}+m_{\theta,\eps}^0+2\eps e^{\lambda t}|x|^2\\
 \le& |(u_1-u_2)^+|_0+|\zeta_1-\zeta_2|_0+\sup_{\a\in\bA}|f_1^\a(\cdot,\cdot,\cdot,0,0)-f_2^\a(\cdot,\cdot,\cdot,0,0)|_{\bar{\cQ}_T\t [-\varphi(|u_2|_0),\varphi(|u_2|_0)]}\\
+& C\theta \sup_{\a\in\bA}\big[|\sigma_1^\a-\sigma_2^\a|_0^2+|b^\a_1-b^\a_2|_0^2+\big|\int_{E} |\eta^\a_1-\eta^\a_2|^2\,(\nu_1\vee\nu_2) (de)\big|_0+\big|\int_{E}(|\eta^\a_1|^2\vee |\eta^\a_2|^2) \,|\nu_1-\nu_2|(de)\big|_0\big]\\
+&\f{C}{\theta}+C\theta +2\eps e^{\lambda t}|x|^2,
\end{align*}
then minimizing the above expression over $\theta$ and passing $\eps\to 0$ lead us to
\begin{align*}
 &u_1(t,x)-u_2(t,x)\le 
 |(u_1-u_2)^+|_0+|\zeta_1-\zeta_2|_0+\sup_{\a\in\bA}|f_1^\a(\cdot,\cdot,\cdot,0,0)-f_2^\a(\cdot,\cdot,\cdot,0,0)|_{\bar{\cQ}_T\t [-\varphi(|u_2|_0),\varphi(|u_2|_0)]}\\
+& C\bigg(\sup_{\a\in\bA}\big[|\sigma_1^\a-\sigma_2^\a|_0^2+|b^\a_1-b^\a_2|_0^2+\big|\int_{E} |\eta^\a_1-\eta^\a_2|^2\,(\nu_1\vee\nu_2) (de)\big|_0+\big|\int_{E}(|\eta^\a_1|^2\vee |\eta^\a_2|^2) \,|\nu_1-\nu_2|(de)\big|_0\big]\bigg)^{\f{1}{4}},
\end{align*}
which enables us to conclude the desired result by using the fact  $\psi(x)=x^{1/4}$ is subadditive.
\end{proof}

\newpage

\end{document}